\RequirePackage{fix-cm}
\documentclass[smallextended]{svjour3}       %
\smartqed  %

\usepackage{graphicx}%
\usepackage{multirow}%
\usepackage{amsmath,amssymb,amsfonts}%
\usepackage{mathrsfs}%
\usepackage[title]{appendix}%
\usepackage{xcolor}%
\usepackage{textcomp}%
\usepackage{manyfoot}%
\usepackage{booktabs}%
\usepackage{algorithm}%
\usepackage{algorithmicx}%
\usepackage{algpseudocode}%
\usepackage{listings}%
\usepackage{subcaption}
\usepackage[square,sort,comma,numbers]{natbib}
\usepackage{doi}
\usepackage{svg}
\usepackage{url}
\usepackage{tikz}
\usetikzlibrary{external}
\ifnum\pdfshellescape=1
\fi
\usepackage{lmodern}
\usepackage{siunitx}
\usepackage[margin=1in]{geometry}

\DeclareRobustCommand{\varlambda}{\text{\usefont{OML}{txmi}{m}{it}\symbol{"15}}}

\begin{document}

\title{Wavelet-based grid adaptation with consistent treatment of high-order sharp immersed geometries\thanks{We acknowledge financial support from an Early Career Award from the Department of Energy, Program Manager Dr.~Steven~Lee, award number DE-SC0020998.}}

\author{Changxiao Nigel Shen \and Wim M.\ van Rees}

\institute{C. N. Shen \at
              Department of Mechanical Engineering, Massachusetts Institute of Technology \\
              \email{shen827@mit.edu.com}           %
           \and
           Wim M.\ van Rees \at
              Department of Mechanical Engineering, Massachusetts Institute of Technology \\
              \email{shen827@mit.edu.com}           %
}
\date{Received: date / Accepted: date}

\maketitle

\abstract{Wavelet-based grid adaptation methods use multiresolution analysis for error estimation, offering a mathematically rigorous approach to adaptive grid refinement when solving Partial Differential Equations (PDEs). However, applying these methods to PDE discretizations with immersed geometries is challenging, as standard interpolating wavelet transforms lose consistency near non-grid-aligned boundary intersections. To address this, we propose a high-order interpolating wavelet transform adaptation strategy compatible with sharp immersed boundary and interface discretizations. The approach performs consistent high-order wavelet transforms on narrow intervals using a 1D polynomial extrapolation technique. To maintain high order, the technique incorporates boundary values and derivatives, which are evaluated from multivariate interpolating polynomials similar to those used in high order immersed finite difference discretizations. Consequently, the proposed approach maintains the wavelet order on any arbitrary smooth multidimensional domain, including near concave geometry sections. This approach enables grid adaptation in complex domains while robustly bounding the numerical error via a manually set refinement threshold. The algorithm’s performance is validated on both static and dynamic problems, including the Navier-Stokes equations with moving boundaries and temporally adapting grid resolutions. The results demonstrate that the proposed method enables effective grid adaptation, establishing a robust, predictable relationship between a user-defined refinement threshold and the overall solution error, even for problems with complex, moving boundaries.}


\maketitle

\section{Introduction}\label{sec1}
Immersed discretization methods such as the immersed boundary method, immersed interface method, or cut-cell finite volume method are prevalent techniques to simulate partial differential equations with complex and/or moving boundaries. Their primary benefit is that they can discretize boundary conditions on arbitrary geometries while using structured background grids (often Cartesian). This avoids the need to create and maintain body-fitted grids that have to adapt to the boundary motion.

When adaptive methods are combined with immersed geometries, the presence of non-grid-aligned boundaries naturally introduces a jump in the field values across the boundary. Smoothness-based indicators on the extended field would thus detect a resolution- and threshold-independent need for refinement. To remedy this, a common practice is instead to fix the spatial resolution of grid cells along the immersed geometry to a preset value, and keep this fixed during the entire simulation. This practice has been adopted independently by many pioneering works on grid adaptation with immersed boundaries/interfaces, such as \cite{popinet2003gerris} using a volume-of-fluid method, \cite{griffith2007adaptive,liska2017fast} using a smooth immersed boundary treatment, and \cite{min2007second} using a level-set based ghost fluid method. This approach may be appropriate for a specific subclass of problems, but general applications are expected to have spatio-temporally varying demands on resolution along the interface. This motivates research on solution-dependent indicators that can account for the discontinuity of the field across the immersed boundary.

Wavelet-based methods offer an alternative approach to grid adaptation \cite{vasilyev2000second, kevlahan2005adaptive, schneider2010wavelet, hejazialhosseini2010high, rossinelli2015mrag}, by relying on a mathematically rigorous multiresolution analysis of a solution field \cite{mallat_multiresolution_1989}. The wavelet collocation method, referred to as the combination of wavelet-based grid adaptation with collocated numerical schemes such as finite difference, offers several advantages over other AMR strategies: it enables high order grid adaptations (i.e., adaptations that have high compression rates and are compatible with high order numerical discretizations), interpretable adaptation thresholds, rigorous mathematical foundations and consistent refinement and coarsening operators. Wavelet collocation methods are frequently combined with  non-body-fitted methods, such as Brinkman penalization \cite{kevlahan2005adaptive,rossinelli2015mrag,engels2019wavelet}, volume of fluid  \cite{fuster2009simulation}, and the immersed interface method \cite{gillis2022murphy,gabbard2024high}. 
To the best of our knowledge, none of these existing wavelet-based adaptive resolution strategies preserves the formal wavelet order in the vicinity of the boundary or interface. This eliminates some of the advantages of the wavelet collocation method, in particular the robust error scaling with user-defined adaptation thresholds.

In this paper, we develop a novel wavelet transform approach for complex domains, which is used within a wavelet-based grid adaptation approach for sharp immersed interface discretizations. The method is based on a consistent interpolating wavelet transform which preserves its formal wavelet order across any smooth irregular-shaped domain and is compatible with high order discrete forcing immersed methods. The contribution of this work is two-fold:
\begin{enumerate}
  \item An interpolating wavelet transform algorithm for any smooth, irregular domains that are discretized by a background Cartesian grid and an immersed interface. This wavelet transform algorithm preserves the wavelet order across the entire domain. Both the forward wavelet transform (FWT) and the inverse wavelet transform (IWT) have computational complexity $\mathcal{O}(N_p)$, where $N_p$ is the total number of grid points inside the irregular domain.
  \item A consistent, temporally adaptive resolution strategy for a high order immersed interface method, based on the proposed wavelet transform algorithm. Numerical evidence is presented that the proposed wavelet collocation scheme achieves a robust bound of numerical error. Indeed, for linear PDEs, we prove that the user-defined refinement threshold $\varepsilon_r$ gives an upper bound to the numerical error of the simulation. Even for nonlinear problems, the proposed strategy is shown empirically to achieve an excellent control of numerical error.
\end{enumerate}

The scope of this paper is restricted to \textit{temporal} grid adaptation, i.e.\ the grids that are spatially uniform but vary their resolution in time. As further discussed at the end of this work, integrating the same algorithm into spatially adaptive resolution schemes requires addressing a non-trivial interdependence in the overlap region, where field points are used for both wavelet transforms and immersed differential operators. 

The remainder of this paper is as follows. In Section \ref{sec:wtai}, we describe our interpolating wavelet transform algorithm in a simple 1D domain. This is followed by an extension to domains of arbitrary shapes in higher dimensions discussed in Section \ref{sec:wtag}. In Section~\ref{sec:coupling} we discuss the integration of the wavelet-based grid adaptation into an IIM solver, and we also perform a detailed mathematical error analysis of the algorithms that we propose. In Section \ref{sec:numresult}, we present our numerical results to support our analysis. Finally, in Section \ref{sec:conclusion}, we draw a conclusion and discuss future directions.

\section{Interpolating Wavelet Transform on Arbitrary Intervals}
\label{sec:wtai}
In this section, we develop a novel framework to extend the interpolating wavelet transform to arbitrary domain geometries by implementing a carefully designed polynomial extrapolation scheme within each wavelet transform step. Since multi-dimensional wavelet transforms are applied as tensorial products of one-dimensional transforms, we first discuss one-dimensional scenarios in this section; the extension to higher dimensions is discussed below in Section~\ref{sec:wtag}.

\subsection{Notation and background}
This work exclusively considers interpolating wavelets, which are based on polynomial interpolation \cite{deslauriers_interpolation_1987}. Suppose we have a grid function $f$ defined on the set of points $x_k^L = k/2^L = x_{2k}^{L+1}$ where $k\in \mathbb{Z}$, and we wish to estimate function values in between these grid points. To obtain an interpolation at the point $x_{k+1/2}^L=(k+1/2)/2^L = x_{2k+1}^{L+1}$, we consider a set of $N=2D$ equispaced points around $k/2^L$, $\chi = \left\{(k-D+1)/2^L,\cdots, (k+D)/2^L\right\}$. The interpolated function value can then be computed as
\begin{equation}
  P_L[f]\left(\frac{k+1/2}{2^L}\right) \equiv p_{k}^L\left(\frac{k+1/2}{2^L}\right), \label{eq:polyinterp}
\end{equation}
where $P_L[f]$ denotes the function interpolant of $f$ based on grid values $x_k^L$, and $p_k^L$ is the unique $(N-1)$-th order polynomial determined by the data tuple $(\chi, f(\chi))$. Applying this recursively, we can obtain the interpolated function value $P_L[f]$ at any binary rational $x$ \cite{donoho1999deslauriers}, and continuously extend the function to any point in $\mathbb{R}$. The above process, known as an inverse wavelet transform, is a linear operator acting on the linear space of the grid values $f(x_k^L)$. In particular, a set of basis functions can be found to represent the interpolant in the interpolation space $V^L$. These basis functions, denoted as $\varphi_k^{L}(x)$, are called the scaling functions, and satisfy
\begin{equation}
  P_L[f](x) \equiv \sum_{k}\lambda_k^L\varphi_k^L(x),
\end{equation} where $\lambda_k^L = f(x_k^L)$ are called the scaling coefficients. The residual of the interpolant at any location $x$ is defined as
\begin{equation}
  R_L[f](x) \equiv P_L[f](x)-f(x).
\end{equation} 
The residuals $R_L[f]$ of all $f$ form another linear space $W^L$, with a set of basis functions $\psi_k^L$ (also known as the wavelet functions), such that
\begin{equation}
  P_{L+1}[f](x) \equiv \sum_{j}\lambda_j^{L+1}\varphi_j^{L+1}(x)= \sum_{k}\lambda_k^{L}\varphi_k^{L}(x) + \sum_{m}\gamma_m^{L} \psi_m^{L}(x).
\end{equation} Here, $\gamma_m^L = R_L[f](x_{m+1/2}^L)$ are the detail coefficients. In summary, the forward wavelet transform decomposes a function $f$ into scaling and detail coefficients $\lambda_k^L, \gamma_m^L$, while the inverse wavelet transform reconstructs the function $f$ from $\lambda_k^L$ and $\gamma_m^L$.

For practical numerical computations, we seldom use the functional analytical framework mentioned above. Instead, we perform the so-called \textit{lifting scheme}, which is mathematically equivalent but much easier to implement numerically \cite{sweldens2005building}. The lifting scheme performs the wavelet transform in three steps. First, the equispaced grid points are split into two sets: the even and the odd points. They are denoted as $\text{even}_{L}=\{x_{2k}^{L+1}\}$ and $\text{odd}_{L}=\{x_{2k+1}^{L+1}\}$, respectively. Next, the function values at even points are used to predict $\tilde{f}$ at the odd points by polynomial interpolation (equation \ref{eq:polyinterp}). The differences between the predictions and the true function values at the odd grid points are then stored as the detail coefficients $\gamma_m^{L}$. The prediction step can be expressed in terms of a filter operation acting on the grid values. This leads to the \textit{dual lifting coefficients}, whose numerical values for wavelet orders $N=2,4,6$ are provided in Table \ref{tab:duallifting}.

\begin{table}[h]
  \centering
  \begin{tabular}{l|llllll}
    N & $S_{-2}$ & $S_{-1}$ & $S_{0}$ & $S_{1}$ & $S_{2}$ & $S_{3}$ \\ \hline
    2 &          &          & 1/2      & 1/2     &         &         \\
    4 &          & -1/16    & 9/16     & 9/16    & -1/16   &         \\
    6 & 3/256    & -25/256  & 75/128   & 75/128  & -25/256 & 3/256
  \end{tabular}
  \vspace{0.4cm}
  \caption{The dual lifting coefficients for $N=2,4,6$, respectively.}
  \label{tab:duallifting}
\end{table}

For a non-lifted interpolating wavelet transform, the scaling coefficients $\lambda^{L}_k$ are equal to the even function values $f(x_{2k}^{L+1})$, and the transform is complete. For an $\tilde{N}$-th order lifted interpolating wavelet transform, the detail coefficients are used to correct the scaling coefficients $\lambda^{L}_k$ so that the $\tilde{N}$-th moment of the function $f$ is preserved. For the interpolating wavelet of order $(N, \tilde{N})$, the $\tilde{N}$-th order dual lifting coefficients are multiplied by $1/2$ for this correction step \cite{sweldens2005building}. The associated \textit{primal lifting coefficients} are provided, e.g., in \cite{gillis2022murphy}. With this notation, the non-lifted interpolating wavelet transform is denoted as $\tilde{N}=0$. For any $\tilde{N}$, the detail coefficients of an $N$th order interpolating wavelet transform on level $L$ satisfy $|\gamma^L| \propto \mathcal{O}(2^{-LN})$ \cite{sweldens1994quadrature}, or, since the grid spacing $h \propto 2^{-L}$, $|\gamma^L| \propto \mathcal{O}(h^N)$.

The inverse wavelet transform is applied through undoing each lifting scheme step of the forward wavelet transform. That is, the scaling coefficients are first "un-corrected" using the detail coefficients; next, function values on the odd points are recovered from the detail coefficients; finally, the even and odd grid points are merged properly.

\subsection{One-dimensional wavelet transform on sufficiently large intervals}
\label{sec:wt1d}
We discuss here applying a wavelet transform on 1D intervals containing a sufficient number of grid points (specified more precisely below). Consider a computational domain $C=[0, 1)$ containing a domain of interest $\Omega=[a, b]$ where $0\leq a<b < 1$ is contained in, but generally does not conform to, the computational domain $C$. The computational domain is initialized at the resolution level $L+1$ by discretizing $C$ into a set of $2^{L+1}$ equally-spaced collocated grid points so that $x_j^{L+1}=j 2^{-(L+1)}$ where $j=0, \cdots , 2^{L+1}-1$. These grid points are further divided into two sets, namely the ones that lie inside our domain of interest: $\mathcal{X}_\Omega^{L+1}=\left\{x_j^{L+1}: a\leq x_j^{L+1} \leq b\right\}$; and those that lie outside: $\mathcal{X}_E^{L+1}=\left\{x_j^{L+1}: 0\leq x_j^{L+1}<a \text{ or } b<x_j^{L+1}< 1\right\}$. In this section, we assume that there are at least $2N$ grid points inside $\Omega$, \textit{i.e.,} $|\mathcal{X}_\Omega^{L+1}|\geq 2N$, where $N$ is the order of the wavelet. The case where $|\mathcal{X}_\Omega^{L+1}|< 2N$ will be discussed in Section \ref{sec:wt1d_narrow}.

We would like to perform forward and inverse interpolating wavelet transforms on the grid points in $\mathcal{X}_\Omega^{L+1}$. We focus on wavelet transforms across one resolution level,
\begin{equation}
  P_{L+1}[f](x) = \sum_{j}\lambda_j^{L+1}\varphi_j^{L+1}(x)\Leftrightarrow \sum_{k}\lambda_k^L\varphi_k^L(x) + \sum_{m}\gamma_m^L \psi_m^L(x)=P_{L}[f](x)+\sum_{m}\gamma_m^L \psi_m^L(x),
\end{equation} 
since wavelet transforms across multiple levels of resolutions are composed of performing the one-level wavelet transform recursively. Before proceeding within this context, we note that interpolating wavelets on intervals with grid-aligned boundaries have been long developed through one-sided interpolation schemes \cite{donoho1992interpolating}. The approach presented below for sufficiently wide intervals is in some sense an extension of this technique to non-grid-aligned intervals, though presented from an extrapolation-interpolation perspective.

\subsubsection{Forward wavelet transform}\label{sec:FWT}
Recall that $L+1$ is the level of resolution associated to the grid $\mathcal{X}_\Omega^{L+1}$. Consider an arbitrary function $f_j=f(x_j^{L+1})$ defined on the set of grid points $x_j^{L+1} \in \mathcal{X}_\Omega^{L+1}$ inside our domain of interest $\Omega$. On this finest level $L+1$, the function values can be reinterpreted as the scaling coefficients (for lifted wavelets, this only holds at the finest level):
\begin{equation}
  \lambda_j^{L+1} = f_j
\end{equation} 
or
\begin{equation}
  P_{L+1}[f](x) = \sum_{j}\lambda_j^{L+1}\varphi_j^{L+1}(x) = \sum_j f_j \varphi_j^{L+1}(x).
\end{equation} 
From here, a one-level forward wavelet transform is performed by explicitly constructing scaling coefficients $\lambda_k^{L}$ and detail coefficients $\gamma_m^L$ at resolution level $L$. For grid points that are well inside $\Omega$, this is done simply by applying the usual wavelet filters (see Table \ref{tab:duallifting}); for a grid point near the boundary, however, the wavelet filters cannot be directly applied. A simple zero-padding across the domain boundary will cause the detail coefficients $\gamma_m^L$ to be of order $\mathcal{O}(1)$, because the extended function is discontinuous across the boundary. This eliminates the advantage of using wavelet detail coefficients as numerical error indicators.
Instead, we follow a strategy based on polynomial extrapolation using the following steps:
\begin{enumerate}
  \item First, split the signal into temporary scaling and detail coefficients on a coarser level $L$. Specifically, define
    \begin{equation}
      \lambda_k^L = \lambda_{2k}^{L+1}, \quad \gamma_m^L = \lambda_{2m+1}^{L+1}.
    \end{equation}
  \item We denote the indices of leftmost and rightmost grid points on level $L$ inside $\Omega$ as $\alpha^L$ and $\beta^L$, respectively. At the left boundary, a polynomial of degree $N-1$ using the scaling coefficients $\lambda_{\alpha^L}^L, \lambda_{\alpha^L+1}^L, \cdots, \lambda_{\alpha^L+N-1}^L$ is constructed. The polynomial is then evaluated at the set of ghost points $x^L_{\alpha^L-1}, \cdots, x^L_{\alpha^L-N/2}$ that are outside but near the left boundary of $\Omega$. A similar extrapolation treatment is done for the right boundary, from which a set of ghost values to the right of the boundary are obtained. This extrapolation strategy is illustrated in Fig. \ref{fig:extrapolation} (top), and is referred to as the Type I extrapolation. An alternative, Type~II extrapolation is applied instead when the nearest-boundary grid point is odd and the boundary function value is known (Fig. \ref{fig:extrapolation}, middle). These two strategies are further discussed in section~\ref{sec:usebc}. \label{sec:extra-without-bc}
  \item Using the scaling coefficients $\lambda_k^L$ together with the extrapolated values $\varlambda_k^L$, a \textit{dual lifting} step is performed. Define the extended field of scaling coefficients as
    \begin{equation}
      \tilde{\lambda}_k^L=
      \begin{cases}
        \lambda_k^L    & \text{if } x^L_k \in \mathcal{X}^L_\Omega \\
        \varlambda_k^L & \text{in ghost regions}                   \\
        0              & \text{otherwise.}
      \end{cases}\label{eq:lambda}
    \end{equation} The \textit{dual lifting} process updates the detail coefficients as follows:
    \begin{equation}
      \gamma_m^L \leftarrow \gamma_m^L - \sum_{k} \tilde{S}_{m, k}\tilde{\lambda}_k^L.
    \end{equation}
    This fully determines the detail coefficients inside the domain $\Omega$. When using a non-lifted interpolating wavelet, the algorithm terminates here.
    \label{sec:duallifting}\\
  \item For lifted interpolating wavelets, an additional \textit{lifting} procedure is needed, where detail coefficients are used to correct the values of scaling coefficients. This amounts to determining the detail coefficients not only inside the domain, but also in a region outside but near the domain boundary, or, the ghost region. Since a $(N-1)$-th order polynomial extrapolation is used to obtain ghost values $\varlambda_k^L$, the detail coefficients in the ghost region are uniformly zero. This is because an $N$-th order interpolating wavelet reconstructs any polynomial of order less than $N$. The extended field of detail coefficients is then defined as
    \begin{equation}
      \tilde{\gamma}_j^L=
      \begin{cases}
        \gamma_m^L & \text{if } x^L_m \in \mathcal{X}^L_\Omega \\
        0          & \text{otherwise}.
      \end{cases} \label{eq:gamma}
    \end{equation}
    The \textit{lifting} procedure then updates the scaling coefficients as follows:
    \begin{equation}
      \lambda_k^L \leftarrow \lambda_k^L + \sum_{m} S_{k, m}\tilde{\gamma}_m^L.
    \end{equation}
    \label{sec:lifting}
\end{enumerate}
After these calculations, all the ghost point values are discarded.

\subsubsection{Inverse wavelet transform}
For the inverse wavelet transform a similar polynomial extrapolation-based approach is adopted. The outline of the algorithm is as follows:
\begin{enumerate}
  \item The first step is to undo the \textit{lifting} procedure. Recall that the detail coefficients in the ghost region are uniformly zero. Therefore, the detail coefficients are extended by zero padding, as described in equation \ref{eq:gamma}. Once this is done, the scaling coefficients update is reverted:
    \begin{equation}
      \lambda_k^L \leftarrow \lambda_k^L - \sum_m S_{k,m} \tilde{\gamma}_m^L.
    \end{equation}
  \item The second step is polynomial extrapolation, using the same Type~I or Type~II strategy as selected in the Forward wavelet transform.
  \item The extended scaling coefficients $\tilde{\lambda}_k^L$ can now be defined using Equation \ref{eq:lambda}. After that, the detail coefficients of the lazy wavelet transform are recovered by performing the \textit{inverse dual lifting}:
    \begin{equation}
      \gamma_m^L \leftarrow \gamma_m^L + \sum_k \tilde{S}_{m,k} \tilde{\lambda}_k^L.
    \end{equation}
  \item The last step is to merge the scaling and detail coefficients $\lambda_k^L$ and $\gamma_m^L$ to obtain $\lambda_j^{L+1}$:
    \begin{equation}
      \lambda_j^{L+1} =
      \begin{cases}
        \lambda_{j/2}^L \quad \text{if } j \text{ is even} \\
        \gamma_{(j-1)/2}^L \quad \text{if }j \text{ is odd}.
      \end{cases}
    \end{equation}
\end{enumerate}
The forward and inverse wavelet transforms preserve the information of the grid values exactly, because each operation that the forward wavelet transform performs on $\lambda_j^{L+1}$ is reverted exactly by an operation in the inverse wavelet transform. 
Further, apart from the cost of constructing polynomial interpolants, both the forward and the inverse wavelet transform algorithms have a computational complexity of order $\mathcal{O}(N_p)$ where $N_p$ is the number of grid points. The total cost of constructing all the polynomial interpolants and obtain ghost point values is $\mathcal{O}(N_c)$ where $N_c$ is the total number of control points.

\subsubsection{The use of boundary conditions} \label{sec:usebc}
The wavelet transform described above, using a Type~I extrapolation strategy, leads to a set of scaling coefficients representing a coarse version of the original field, and a set of detail coefficients that decay at the order $\mathcal{O}(h^N)$ and encode the local smoothness of the field. Near the boundary, the detail coefficients will generally be larger than those in the free space \cite{harnish2022adaptive}, given that the function has relatively uniform smoothness throughout the entire domain. This is the well-known Runge phenomenon, and is most severe when an immediate neighboring grid point of the boundary ($\alpha^{L+1}$ or $\beta^{L+1}$) is odd. In this case, the computed detail coefficient at the corresponding grid point, $\gamma_{(\alpha^{L+1}-1)/2}^L$ or $\gamma_{(\beta^{L+1}-1)/2}^L$ will be amplified by a Lebesgue constant of magnitude $\mathcal{O}(2^N)$ compared to its free space counterpart, as detailed in the Appendix section~\ref{sec:detailwall}.

In some applications of the wavelet transform, such as when solving numerical PDEs, a Dirichlet boundary condition might be given along with the solution field. In this case, the boundary condition is incorporated into the wavelet transform to reduce the magnitude of the detail coefficients near the boundary. For example, assume that $\alpha^{L+1}$ is odd, and the boundary condition $f(a)$ is given. The extrapolation operator in Step \ref{sec:extra-without-bc} of the forward wavelet transform is then modified by incorporating the boundary value $f(a)$. Specifically, a polynomial of degree $N-1$ is constructed using $f(a)$ and the scaling coefficients
$\lambda_{\alpha^L}^L, \lambda_{\alpha^L+1}^L, \cdots, \lambda_{\alpha^L+N-2}^L$. When a Neumann boundary condition is imposed, it can first be converted into Dirichlet BC using a polynomial interpolant \cite{gabbard2022immersed}, followed by exactly the same procedure described above.

We refer to this extrapolation scheme with boundary condition as Type II extrapolation, which is illustrated in Figure \ref{fig:extrapolation} (middle). Type II extrapolation is only used when the closest grid point to the immersed boundary or interface is odd; when it is even, the boundary intersection location and the grid point can be arbitrarily close to each other, leading to a high condition number of the Vandermonde matrix and hence numerical instability. Moreover, the Runge phenomenon in this situation is not as severe for moderate wavelet order $N$, so Type I extrapolation is sufficient in most practical situations.

Figure \ref{fig:compare_wavelets} compares the scaling functions near the boundary that correspond to Type I and Type II extrapolation, respectively. The functions are generated by using the definition of interpolating wavelets -- through recursive polynomial interpolation, as described in \cite{donoho1999deslauriers, sweldens2005building}.

\begin{figure}[htbp]
  \centering

  \begin{subfigure}[t]{0.48\linewidth}
    \centering
    \includegraphics[width=\textwidth]{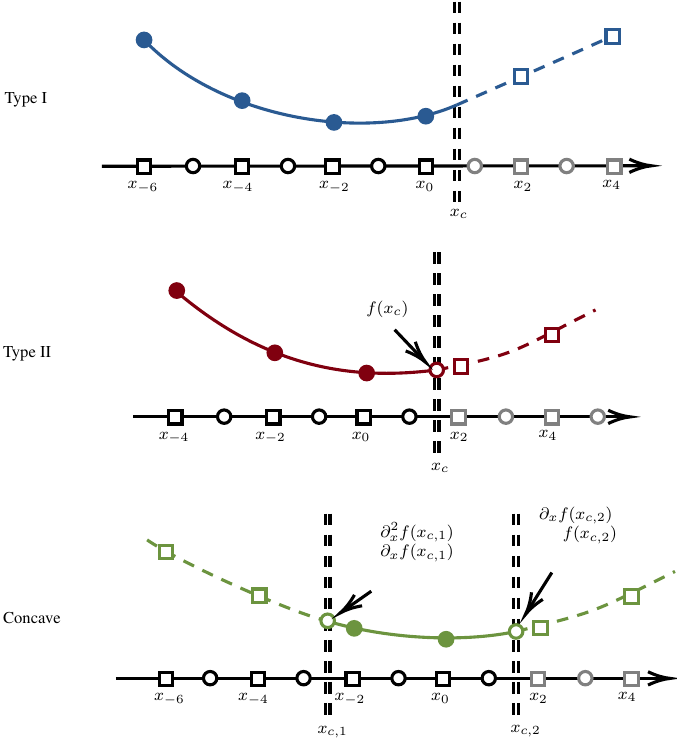}
    \caption{Extrapolation schematics}
    \label{fig:extrapolation}
  \end{subfigure}\hfill
  \begin{subfigure}[t]{0.48\linewidth}
    \centering
    \includegraphics[width=\linewidth]{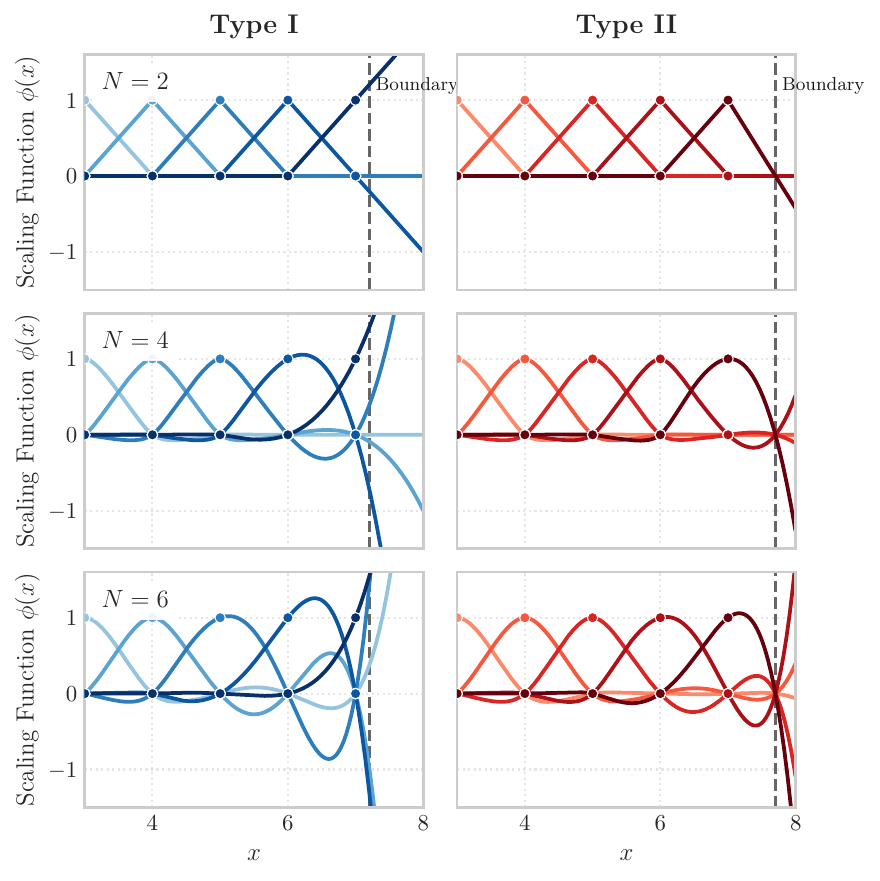}
    \caption{Arbitrary-boundary interpolation}
    \label{fig:compare_wavelets}
  \end{subfigure}

  \caption{Left: Three types of extrapolation strategies in the proposed wavelet transform algorithm. For a sufficient number of grid points, when the immediate neighbor of the control point has an even (odd) index, Type I (Type II) extrapolation is used. For an insufficient number of grid points (concave sections in $>1$D), a Hermite-like extrapolation is used. Right: Modified scaling functions near the boundary for Type I (blue) and Type II (red) extrapolations, for different wavelet orders $N$ (top to bottom).}
\end{figure}

\subsection{One-dimensional wavelet transforms on narrow intervals}
\label{sec:wt1d_narrow}
We now consider the scenario where the 1D interval contains fewer than $2N$ grid points inside $\Omega$, \textit{i.e.,} $|\mathcal{X}_\Omega^{L+1}|< 2N$. Such cases readily occur in higher dimensions when considering concave geometry sections, as sketched in Fig.~\ref{fig:concave} below. In such cases one can either revert to a lower-order wavelet transform following the same principles as sketched above, which would imply an inconsistency in grid adaptation thresholds between interior and near-boundary regions. Here instead high-order grid adaptation is maintained using Hermite-like interpolants to construct the ghost point values needed across the boundary. The Hermite-like polynomials require information on the function value and derivative(s) at the immersed boundary; details on how those are constructed are explained in the next section. For now, we assume the function values and derivative(s) exist at the boundary.

As an example, consider the case sketched in Fig.~\ref{fig:extrapolation} (bottom). In this case, there are two available grid points between $x_{c,1}$ and $x_{c,2}$. Suppose the wavelet is sixth order; then a fifth degree polynomial needs to be constructed to consistently extrapolate ghost values. In this case, $6-2=4$ degrees of freedom remain, which need to be filled using boundary values and derivatives of $f$ along the $x$-axis evaluated at the control points. Since in this example $x_{-2}$ is an even index, the boundary value $f(x_{c,1})$ is not used; on the right, $x_1$ is odd so $f(x_{c,2})$ is used. The remaining required function derivatives are evenly split between two control points. In the case shown in Figure \ref{fig:extrapolation}, this implies $\partial_x f(x_{c,1})$ and $\partial^2_x f(x_{c,1})$ are required at the left control point, and $f(x_{c,2})$ and $\partial_x f(x_{c,2})$ are required at the right control point. For 1D cases, estimating these boundary values and derivatives to the required order of accuracy from the interior function values poses the same challenges as constructing the interpolating wavelet polynomial itself. However, in higher dimensions multivariate polynomials are used to estimate the boundary derivatives, which will be further discussed in Section~\ref{sec:concave} below.

For now, we assume that the required boundary values/derivatives are known. We then proceed by constructing a one-dimensional Hermite-like interpolant $q(x)$ of polynomial degree $N-1$, which satisfies the constraints based on available grid values and the field derivatives on the boundary. For example, in the case shown in Figure \ref{fig:extrapolation} (bottom), the 1D interpolant needs to satisfy
\begin{equation}
  \begin{cases}
    q(x_{-2}) = f(x_{-2})                 \\
    q(x_0) = f(x_{0})                     \\
    q'(x_{c, 1}) = \partial_x f(x_{c, 1})    \\
    q''(x_{c, 1}) = \partial_x^2 f(x_{c, 1}) \\
    q(x_{c, 2}) = f(x_{c, 2})                \\
    q'(x_{c, 2}) = \partial_x f(x_{c, 2}),    
  \end{cases}
\end{equation}
These six conditions, provided that they are linearly independent, determine a unique fifth degree polynomial $q(x)$. Ghost values outside of the domain are then obtained by evaluated $q(x)$ at the ghost locations. Notice that the ghost point values to the left of $x_{c,1}$ and those to the right of $x_{c,2}$ come from the same polynomial interpolant $q(x)$. This is a crucial step as it guarantees that the detail coefficients $\gamma^L$ in the ghost region are uniformly zero, thus maintaining a lossless wavelet transform for all field values.

With the ghost values constructed, the 1D \textit{lifting} and \textit{dual lifting} steps described above in Section \ref{sec:FWT} are performed.

\section{Interpolating Wavelet Transforms on Arbitrary 2D Domains} \label{sec:wtag}
In this section we discuss how to perform forward and inverse wavelet transforms in complex 2D domains. Convex domain sections, which are handled through straightforward tensor products of `sufficiently wide` 1D wavelet transforms, are discussed first. We then treat concave domain sections, where narrow intervals may occur that require boundary function derivatives.

\subsection{Convex geometries}\label{sec:dimsplit}
We first focus on the simple case of convex boundaries. Let $C=[0, 1)^2$ be our computational domain with periodic boundary conditions. Let $E\subset C$ be a convex set in this domain.
The field $f(\mathbf{x})$ is defined on $\Omega = C \backslash E$, and the goal is to perform a consistent two-dimensional wavelet transform on the field $f$.

First, the computational domain $C$ is discretized using a Cartesian grid with grid spacing $h = 2^{-(L+1)}$. The grid points are divided into two sets:
\begin{align}
  \mathcal{X}_{\Omega}^{L+1} & = \left\{\mathbf{x}_{i,j}: \mathbf{x}_{i,j} \in \Omega\right\}, \\
  \mathcal{X}_{E}^{L+1}      & = \left\{\mathbf{x}_{i,j}: \mathbf{x}_{i,j} \in E\right\}.
\end{align}
In addition, the \textit{control points} are defined as the intersections of the domain boundary $\Gamma = \partial \Omega$ and the Cartesian grid lines, which are denoted as
\begin{equation}
  \mathcal{X}_c^{L+1} = \left\{\mathbf{x}_c : \mathbf{x}_c = (x_c, y_c)\in \partial \Omega \text{ and } (x_c=ih \text{ or } y_c = ih, i\in\mathbb{N})\right\}.
\end{equation}

Wavelet transforms in higher dimensions are simply tensor products of the 1D wavelet transforms, which motivates a dimension-split approach to treating immersed geometries. First, along each horizontal grid line (in the $x$ direction), a 1D wavelet transform is performed based on the scheme proposed in Section \ref{sec:FWT}. The role of control points in 2D is exactly the same as the role of the boundary points in performing the one-dimensional wavelet transform. Further, because the domain is convex, each 1D wavelet transform adjacent to a boundary is guaranteed to have a sufficient number of grid points to construct the Type~I or Type~II polynomials as described in Section~\ref{sec:usebc}. After performing the $x$-direction wavelet transforms, the result is still a two-dimensional field, with scaling coefficients stored in locations $\{\mathbf{x}_{i,j}: i \text{ is even and } \mathbf{x}_{i,j}\in \mathcal{X}_{\Omega}\}$, and detail coefficients stored at locations $\{\mathbf{x}_{i,j}: i \text{ is odd and } \mathbf{x}_{i,j}\in \mathcal{X}_{\Omega}\}$. Subsequently, a second 1D wavelet transform is performed along each vertical grid line (in the $y$ direction). Vertical grid lines are classified into those storing $x$-scaling coefficients and those storing $x$-detail coefficients. For a smooth field $f \in \mathcal{C}^\infty$, the values of the $x$-detail coefficients are already of magnitude $\mathcal{O}(h^N)$, where $N$ is the order of the wavelet. Therefore, it is not necessary to perform polynomial extrapolations along the grid lines that store $x$-detail coefficients. Instead, when encountering the domain boundary along those grid lines, the ghost points on the other side of the interface are simply filled with zero values. For grid lines that store $x$-scaling coefficients, ghost point values are filled using the 1D polynomial extrapolation approach as described in Section \ref{sec:FWT}. After both 1D wavelet transforms, the 2D field is filled with the following data:
\begin{enumerate}
  \item $\lambda^{L}$, which are defined on the set $\{\mathbf{x}_{i,j}: \mathbf{x}_{i,j}\in \mathcal{X}_\Omega^{L+1}, i,j \text{ are even}\}$;
  \item $\gamma_x^{L}$, which are defined on the set $\{\mathbf{x}_{i,j}: \mathbf{x}_{i,j}\in \mathcal{X}_\Omega^{L+1}, i \text{ is odd, }j \text{ is even}\}$;
  \item $\gamma_y^{L}$, which are defined on the set $\{\mathbf{x}_{i,j}: \mathbf{x}_{i,j}\in \mathcal{X}_\Omega^{L+1}, i\text{ is even, }j \text{ is odd}\}$;
  \item $\gamma_{x,y}^{L}$, which are defined on the set $\{\mathbf{x}_{i,j}: \mathbf{x}_{i,j}\in \mathcal{X}_\Omega^{L+1}, i,j \text{ are odd}\}$.
\end{enumerate}
This completes the wavelet transform on convex domains.

\begin{figure}[htbp]
  \centering
    \includegraphics[width=1\linewidth]{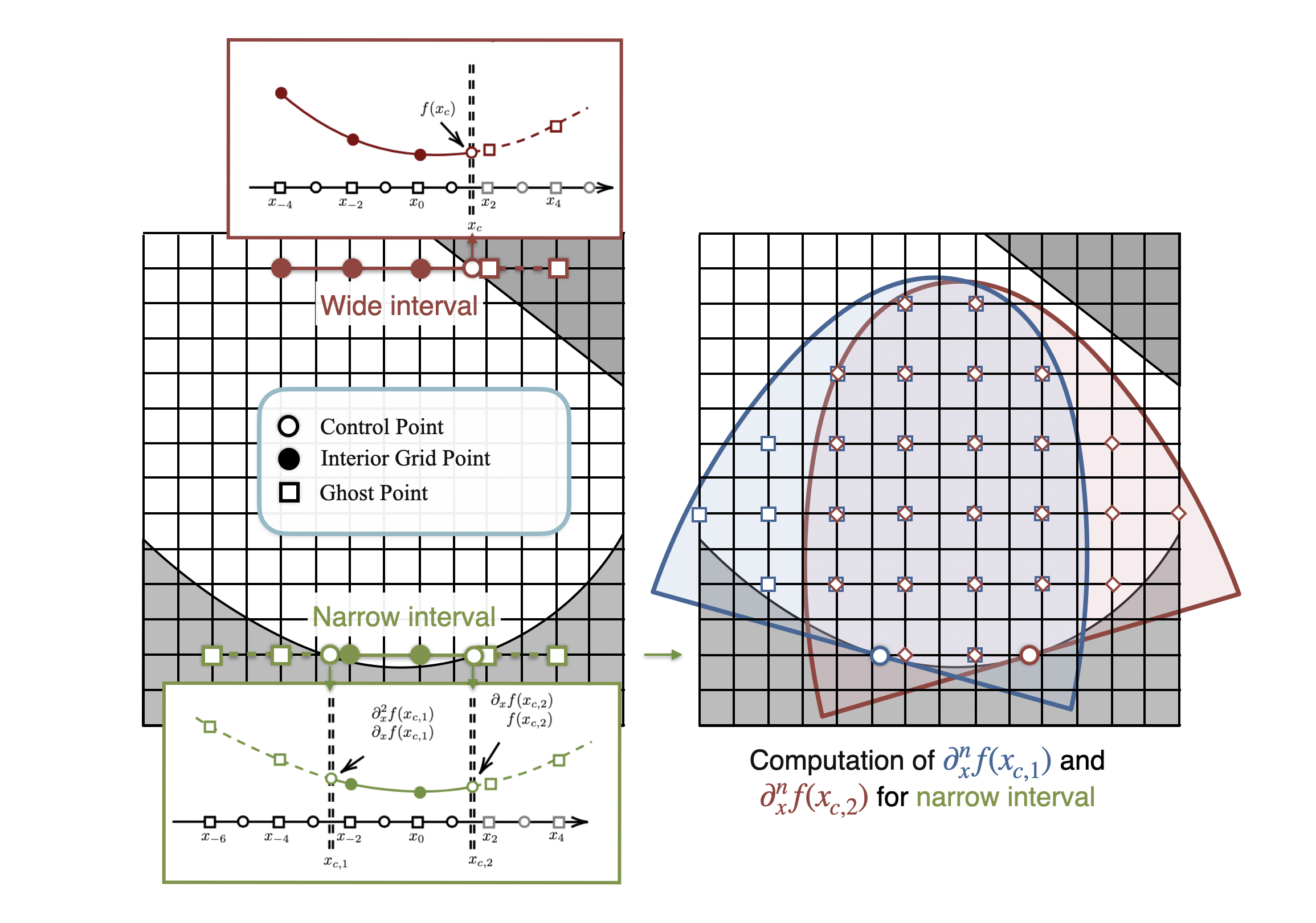}
  \caption{Left: When the boundary is convex (top, red line) a `wide interval' wavelet transform can be performed. For a concave section (bottom, green) there might be insufficient grid points to apply the 1D extrapolation stencil along a grid line. Right: when such a concavity occurs, we employ the half-elliptical least square fit strategy. For each control point elliptical stencils are used to construct polynomial approximations to the high order derivatives on the boundary. These derivatives are in turn used to perform a Hermite-like 1D polynomial extrapolation along the grid line, which enables the 1D wavelet transform on the small interval between the two control points.}
    \label{fig:concave}  
\end{figure}

\subsection{Concave geometries}\label{sec:concave}
For concave geometries in 2D or 3D, there may not be sufficient grid points along a grid line to perform polynomial extrapolations or wavelet transforms to the required order of accuracy. Section~\ref{sec:wt1d_narrow} describes how, for such narrow intervals, a Hermite-like polynomial is constructed to fill in ghost points on either side of the interface. Here we discuss how the required boundary function derivatives for constructing such polynomials are obtained in higher dimensions, relying heavily on the immersed interface framework described in \cite{gabbard2022immersed, gabbard2023high, gabbard2024high, gabbard2025high}.

First, in general, we need to detect concavity and pair control points. For simplicity, we focus on an $x$-direction wavelet transform in this section, so that each control point $\mathbf{x}_{c, 1}$ is located on a horizontal grid line. Along that grid line, the number of grid points that are available for a 1D $N-1$-th degree polynomial extrapolation are counted. If the number of available grid points is less than $N$, additional boundary information is required to perform the wavelet transform. Moreover, in this case an adjacent control point $\mathbf{x}_{c, 2}$ is located on the same grid line, with $||\mathbf{x}_{c,1}-\mathbf{x}_{c,2}||<2Nh$ where $h$ is the grid spacing. The existence of such an $\mathbf{x}_{c,2}$ is guaranteed by the continuity of the boundary $\partial \Omega$.

Next, partial derivatives $\partial_x^j f(\mathbf{x}_{c,1})$ and $\partial_x^j f(\mathbf{x}_{c,2})$ are computed on each pair of control points $(\mathbf{x}_{c,1}, \mathbf{x}_{c,2})$ associated with a `narrow interval'. The number of higher order derivatives that are needed depends on the geometry as well as the order of interpolation, as described in Section~\ref{sec:wt1d_narrow} above. To evaluate each one of these partial derivatives, a multivariate least squares polynomial fit is used, similar to the IIM finite-difference discretization proposed in \cite{gabbard2023high, gabbard2024high}. To do so, each control point relies on the set of grid points that are (1) within $\Omega$, and (2) fall within the region of a half ellipsoid centered at the control point with a major radius $hr_n$ along the interface normal and a minor radius $hr_t$ along the tangential direction of the boundary (Figure \ref{fig:concave}, right). These grid points are selected using the strategy proposed in \cite{gabbard2023high}. For each control point $\mathbf{x}_c$, we denote the distance between it and an arbitrary point $\mathbf{x}$ as $\boldsymbol{\xi}_c(\mathbf{x}) = \mathbf{x} - \mathbf{x}_c$. A distorted distance function is then defined as

\begin{equation}
  d_c(\mathbf{x}) = \sqrt{\boldsymbol{\xi}^T_c(x)\Sigma \boldsymbol{\xi}_c(x)}\quad \text{with}\quad \Sigma = \frac{1}{h^2 r_n^2} \mathbf{n}_c \mathbf{n}_c^T + \frac{1}{h^2 r_t^2}(I-\mathbf{n}_c \mathbf{n}_c^T).
\end{equation}

This function $d_c(\mathbf{x})$ ranges from zero at the control point to one at the edge of the ellipsoid. This leads to the definition of the following set:

\begin{equation}
  \mathcal{E}_c^{L+1} = \left\{\mathbf{x}_{i,j} \in \mathcal{X}^{L+1}_{\Omega}: \mathbf{\xi}_c (\mathbf{x}_{i,j})\cdot \mathbf{n}_c\geq 0, d_c(\mathbf{x}_{i,j})\leq 1, \text{and } i,j \text{ are even.}\right\}
\end{equation}

Using the set of points defined above, which are essentially grid points that fall within the half ellipsoid and whose indices $i, j$ are even,  a high order least square polynomial fit to our original field $f$ is constructed:
\begin{equation}
  p(\mathbf{x}) = \arg \min_{g(\mathbf{x})\in \mathcal{P}_{2, N-1}} \sum_{\mathbf{x}_{i,j}\in \mathcal{E}_c^{L+1}}\left(g(\mathbf{x}_{i,j})-f(\mathbf{x}_{i,j})\right)^2
\end{equation}
where $\mathcal{P}_{2, N-1}$ denotes the space of polynomials of 2 variables that are of degree less than or equal to $N-1$. As long as the basis of $\mathcal{P}_{2, N-1}$ is linearly independent over the set $\mathcal{E}_c^{L+1}$, this interpolant $p(\mathbf{x})$ is unique and satisfies
\begin{equation}
  \partial^{\boldsymbol{\alpha}}p(\mathbf{x}) = \partial^{\boldsymbol{\alpha}}f(\mathbf{x})+\mathcal{O}\left(h^{N-|\boldsymbol{\alpha}|}\right)
\end{equation}
for any evaluation point $\mathbf{x}$ and multi-index $\boldsymbol{\alpha}$. Therefore,  the $x$-derivatives of $p(\mathbf{x})$ approximate the derivatives of $f(\mathbf{x})$ at the control point $\mathbf{x}_c$ with the following relation:
\begin{equation}
  \partial^{n}_xp(\mathbf{x}_c) = \partial^{n}_xf(\mathbf{x}_c)+\mathcal{O}\left(h^{N-n}\right). \label{eq:hermiteerror}
\end{equation}

Using the specialized elliptical stencils for control points $\mathbf{x}_{c,1}$ and $\mathbf{x}_{c,2}$, the required boundary values and derivatives are estimated as $\partial_x^j f(x_{c,1})\approx \partial_x^j p_1(\mathbf{x}_{c,1})$ and $\partial_x^j f(\mathbf{x}_{c,2})\approx \partial_x^j p_2(\mathbf{x}_{c,2})$. Here $p_1$ and $p_2$ are generally two different polynomials, because two different elliptical stencils are used centered on $\mathbf{x}_{c,1}$ and $\mathbf{x}_{c,2}$ respectively. Nevertheless, as noted in Section~\ref{sec:wt1d_narrow} above, after constructing the Hermite-like polynomial, the ghost point values to the left of $\mathbf{x}_{c,1}$ and those to the right of $\mathbf{x}_{c,2}$ come from the same polynomial interpolant $q(x)$. This guarantees that the detail coefficients outside of the domain remain zero and the wavelet transform is lossless.

\section{Temporal grid adaptation for PDEs in complex domains} \label{sec:coupling}
Here we describe how to couple the proposed interpolating wavelet transform on arbitrary 2D domain geometries with PDE solvers that are based on sharp immersed discretizations. This coupling enables the solver to choose a proper level of resolution based on the smoothness of the solution field, hence maximizing the efficiency while targeting an accuracy based on a user-determined threshold. The fully coupled algorithm is described as follows.

Denote the solution field at resolution level $L$ as $u^L(\mathbf{x}, t)$, with the initial resolution $L_0=L(t=0)$. Before initiating the solver, a coarsening threshold $\varepsilon_c>0$ and a refinement threshold $\varepsilon_r> \varepsilon_c$ are chosen. Every $n$ time steps,  a forward wavelet transform is performed on $u^L(\mathbf{x}, t)$, yielding $\lambda^{L-1}(\mathbf{x}, t)$ and $\gamma^{L-1}(\mathbf{x}, t)$. The field is then coarsened or refined based on the following criteria (with $k$ a parameter discussed below):
\begin{enumerate}
  \item If $||\gamma^{L-1}(\mathbf{x}, t)||_\infty < 2^{-k(L-L_0)}\varepsilon_c $, the solution field is coarsened by setting $u^{L-1}(\mathbf{x}, t)=\lambda^{L-1}(\mathbf{x}, t)$;
  \item If $||\gamma^{L-1}(\mathbf{x}, t)||_\infty \ge 2^{-k(L-L_0)}\varepsilon_r$, the field is refined by setting $\gamma^L(\mathbf{x}, t)=0$, $\lambda^L(\mathbf{x}, t)=u^L(\mathbf{x}, t)$, and performing an inverse wavelet transform to obtain $u^{L+1}(\mathbf{x}, t) = \lambda^{L+1}(\mathbf{x}, t)$;
  \item If $ 2^{-k(L-L_0)}\varepsilon_c \le ||\gamma^{L-1}(\mathbf{x}, t)||_\infty < 2^{-k(L-L_0)}\varepsilon_r$, the field stays at the current resolution level $L$.
\end{enumerate}

Preventing `flip-flopping', i.e.\ successive alternating compression/refinement of the same region without significant changes in the field, is achieved by ensuring $\varepsilon_r \geq 2^N \varepsilon_c$, with $N$ the order of the wavelet transform \cite{gillis2022murphy}.

The introduction of the parameter $k$ above leads to a family of adaptation strategies. It is common to set $k=0$, which gives a level-independent adaptation criterion that is robust in practice. However, a formal analysis shows that taking $k$ equal to the spatial order of the PDE to be solved (level-dependent strategy), leads to an explicit relationship between the refinement threshold $\varepsilon_r$ and the numerical error made during the simulation for linear PDEs.  This is proved with the following theorem.

\begin{proposition} \label{prop:3}
  Suppose a linear, well-posed $k$-th order PDE
  \begin{equation}
    u_t=\mathcal{L}u
  \end{equation} is semi-discretized by an $M$-th order, stable, immersed interface finite difference scheme. If the exact solution $u$ is $C^\infty$ smooth, then the wavelet collocation adaptive resolution scheme with $N = M+k$-th order wavelet and level-dependent adaptation strategy controls the numerical error at time $T$ by
  \begin{equation}
    ||u-u_h||_\infty (T)\leq \mathcal{C}[u](T) \varepsilon_r
  \end{equation} where $\mathcal{C}[u](T)$ is a constant determined only by the exact solution $u$ and time $T$, and $\varepsilon_r$ is the refinement threshold. \label{thm:errorcontrol}
\end{proposition}
\textbf{Proof:} Denote $\mathcal{L}^{L(t)}$ as the discretization of the linear operator $\mathcal{L}$ at the resolution level $L$. Notice that since the spatial resolution changes adaptively over time, $L(t)$ and thus the discretization $\mathcal{L}^{L(t)}$ themselves are functions of time $t$. Let $u_h^{L(t)}$ denote the numerical solution defined on the grid points of resolution level $L$.

The numerical error on the entire domain is defined as
\begin{equation}
  e(\mathbf{x}, t)=u(\mathbf{x}, t) - \tilde{u}_h(\mathbf{x}, t).
\end{equation} Here $\tilde{u}_h$ denotes the interpolant of $u_h$ that gives the smallest pointwise error against the exact solution $u$, so that 
\begin{equation}
  ||e(\mathbf{x}, t)||_\infty \sim \max_{i}|u(\mathbf{x}_i, t)-u_h(\mathbf{x}_i, t)|,
\end{equation} where $\mathbf{x}_i$ are the discrete grid points.

Next, split up the time interval $[0, T]$ into several subintervals $\mathcal{T}_1 = [0, t_1], \mathcal{T}_2 = [t_1, t_2], \cdots , \mathcal{T}_n = [t_{n-1}, t_n]$ where $t_n=T$. On each interval $\mathcal{T}_i$, the spatial resolution level $L(t)=L_i$ remains constant. The proposition can then be proved by induction: assume that the statement is true up to time $t_i$, i.e.\
\begin{equation}
  ||e^{L_i}(\cdot, t_i)||_\infty \leq \mathcal{C}(t_i) \varepsilon_r. \label{eq:errorbound}
\end{equation}
Given that $e(\cdot, 0) = 0$, the only step left is to show that on the interval $[t_i, t_{i+1}]$, the numerical error remains bounded by $\mathcal{C}[u](t) \varepsilon_r$. This is shown by a standard \textit{a posteriori} error analysis below.

The evolution of the numerical error on the time interval $\mathcal{T}_{i+1}$ satisfies
\begin{align}
  \begin{split}
    \frac{\mathrm{d}}{\mathrm{d}t}e(\cdot, t) & = \mathcal{L}u - \mathcal{L}^{L_{i+1}}u_h^{L_{i+1}}                                                   \\
    & =\mathcal{L}u - \mathcal{L}\tilde{u}_h + \mathcal{L}\tilde{u}_h - \mathcal{L}^{L_{i+1}}u_h^{L_{i+1}}  \\
    & =\mathcal{L}(u-\tilde{u}_h)+\left(\mathcal{L}\tilde{u}_h - \mathcal{L}^{L_{i+1}}u_h^{L_{i+1}}\right).
  \end{split} \label{eq:errorde}
\end{align} Since $\tilde{u}_h(\mathbf{x}_i, t)=u_h^{L_{i+1}}(\mathbf{x}_i, t)$ on all the grid points $\mathbf{x}_i$,
\begin{equation}
  \mathcal{L}\tilde{u}_h - \mathcal{L}^{L_{i+1}}u_h^{L_{i+1}} = \mathcal{L}\tilde{u}_h - \mathcal{L}^{L_{i+1}}\tilde{u}_h = \tilde{\tau}_h
\end{equation} 
is the local truncation error of the interpolant $\tilde{u}_h$. Substituting the definition of $\tilde{\tau}_h$ in to the differential equation of the numerical error, and solving, yields
\begin{equation}
  e(\cdot, t)= \exp\left(\mathcal{L}(t-t_i)\right)e(\cdot, t_i)+\int_{t_i}^t \exp(\mathcal{L}(t-s))\tilde{\tau}_h(\cdot, s) \mathrm{d}s.
\end{equation}
Taking the norm of both sides leads to
\begin{align}
  \left|\left|e(\cdot, t)\right|\right|_\infty \leq & \qquad \left|\left|\exp(\mathcal{L}(t-t_i))\right|\right|_{\infty}||e(\cdot, t_i)||_\infty \nonumber                                                   \\
  & +\int_{t_i}^t \left|\left|\exp \left(\mathcal{L}(t-s)\right)\right|\right|_\infty ||\tilde{\tau}_h(\cdot, s)||_\infty \mathrm{d}s \label{eq:errbound1}
\end{align}
For any well-posed PDE, we assume that $\max \Re[\mathrm{eig}(\mathcal{L})]\leq 0$. This implies that there exists constants $\omega \leq 0$ and $c\geq 0$ such that
\begin{equation}
  \left|\left|\exp\left(\mathcal{L}t\right)\right|\right|_\infty \leq c\exp(\omega t)
\end{equation} 
where $c$ and $\omega$ are independent of time and mesh size $h$.

By standard truncation error analysis, the following holds
\begin{equation}
  \tilde{\tau}_h(\mathbf{x}_i, t) = C_\tau \tilde{u}_h^{N}(\mathbf{x}_i, t)h^{M} + \mathcal{O}(h^{M+1}). \label{eq:boundtau1}
\end{equation}
Meanwhile, as can be shown by a Taylor expansion, the detail coefficients of an $N$-th order wavelet can be written as
\begin{equation}
  \gamma^{L_{i+1}}(\mathbf{x}_i, t) = C_\gamma \tilde{u}_h^{N}(\mathbf{x}_i, t)h^{N} + \mathcal{O}(h^{N+1}).
\end{equation} 
Since $||\gamma^{L_{i+1}}||_\infty\leq 2^{-k(L_{i+1}-L_0)}\varepsilon_r=C_0 \varepsilon_r h^k$, it is concluded that
\begin{equation}
  ||C_\gamma \tilde{u}_h^{N}(\mathbf{x}_i, t)h^{M} + \mathcal{O}(h^{M+1})||_\infty \leq C_0 \varepsilon_r. \label{eq:boundtaylor}
\end{equation}
Combining Equation \ref{eq:boundtaylor} with Equation \ref{eq:boundtau1} and dropping high order terms gives
\begin{equation}
  \tilde{\tau}_h(\mathbf{x}_i, t) \leq C\varepsilon_r
\end{equation} 
for some constant $C$. Finally, from our initial assumption the following holds: $||e(\cdot, t_i)||_\infty \leq \mathcal{C}(t_i)\varepsilon_r$. Substituting these upper bounds into Equation \ref{eq:errbound1} yields
\begin{align}
  ||e(\cdot, t)||_\infty &\leq c\exp(\omega(t-t_i)) \mathcal{C}(t_i)\varepsilon_r + c C\varepsilon_r \int_{t_i}^t \exp(\omega(t-s)) \mathrm{d}s
\end{align}
If $\omega < 0$, this gives
\begin{align}||e(\cdot, t)||_\infty\le  \underbrace{\left(\exp(\omega(t-t_i))\mathcal{C}(t_i)+\frac{\exp(\omega(t-t_i))-1}{\omega}C\right)c}_{\mathcal{C}(t)}\ \varepsilon_r;
\end{align}
Otherwise if $\omega=0$, then
\begin{equation}
  ||e(\cdot, t)||_\infty\leq \underbrace{\left(\mathcal{C}(t_i) + (t-t_i)C\right)c}_{\mathcal{C}(t)}\ \varepsilon_r.
\end{equation}\\

With this, it is concluded that at time $t_{i+1}$, the numerical error is formally bounded by
\begin{equation}
  ||u-u_h||_\infty (t_{i+1})\leq \mathcal{C}(t_{i+1}) \varepsilon_r
\end{equation}
where $\mathcal{C}[u](t_{i+1})$ is a constant determined by the exact solution $u$ and time $t_{i+1}$. By induction, since $e(\cdot, 0) = 0$, the proposition holds for an arbitrary time $T$. $\Box$

\section{Numerical results} \label{sec:numresult}
In this section, we present the results and analysis of the numerical experiments for the proposed immersed interpolating wavelet transform algorithm. We first focus on a static field $f(\mathbf{x})$ and its wavelet transforms in an irregular, nonconvex domain in Section \ref{sec:epstest}. Then, we couple our wavelet transform with a finite difference immersed interface method developed in \cite{gabbard2023high, gabbard2024high} for adaptive-grid simulations of various PDEs with immersed bodies.
Examples are drawn from the heat equation (Section \ref{sec:parabolic}), as well as the Navier-Stokes equations (Section \ref{sec:NS}).

\subsection{Compression of a static field}\label{sec:epstest}
Consider a static sine wave field defined as
\begin{equation}
  f(x,y) = 100\sin(4\pi x)\sin(4\pi y).
\end{equation} 
The field is defined on a unit square domain masked out by a star shape $\Omega=[0,1]^2 \backslash E$. The boundary of the star shape is defined through a level set function
\begin{equation}
  \phi(r, \theta) = r-(0.3+0.04\sin(5\theta))
\end{equation}
where $r=\sqrt{(x-x_0)^2+(y-y_0)^2}$, $\theta=\arctan \left(\frac{y-y_0}{x-x_0}\right)$, and $[x_0, y_0]=[0.51, 0.51]$ is the center of the star.

We adopt the same procedure as in \cite{gillis2022murphy} to perform the compression test. Specifically, the field $f(x, y)$ is first initialized on the finest grid at resolution level $L_{\max}$, which is chosen to be a $1024\times 1024$ Cartesian grid. Four consecutive forward wavelet transforms are then applied to coarsen the field to resolution $64 \times 64$. At each resolution level, all the detail coefficients whose magnitude is below a given compression threshold $\varepsilon$ are discarded by setting them to zero. Finally, a hierarchy of inverse wavelet transforms is performed to obtain the compressed field at the maximum resolution level $L_{\max}$, denoted $P^{L_{\max}}[f]_{\varepsilon}(x, y)$. The infinite norm error of the compressed field $P^{L_{\max}}[f]_{\varepsilon}(x, y)$ compared to the original field $f(x,y)$ is defined and computed as:
\begin{equation}
  E_{\infty}(\varepsilon)=\left\lVert f(x,y)-P^{L_{\max}}[f]_{\varepsilon}(x, y) \right\rVert_{\infty}.
\end{equation}
Based on wavelet theory \cite{donoho1992interpolating, vasilyev2000second}, a linear relationship between $E_{\infty}$ and $\varepsilon$ is expected, so that $E_{\infty}=C\varepsilon$. Typically, the constant $C$ is of magnitude $\mathcal{O}(1)$.

\begin{figure}[t]
  \begin{subfigure}[c]{0.32\textwidth}
    \centering{\includegraphics[width=0.99\linewidth]{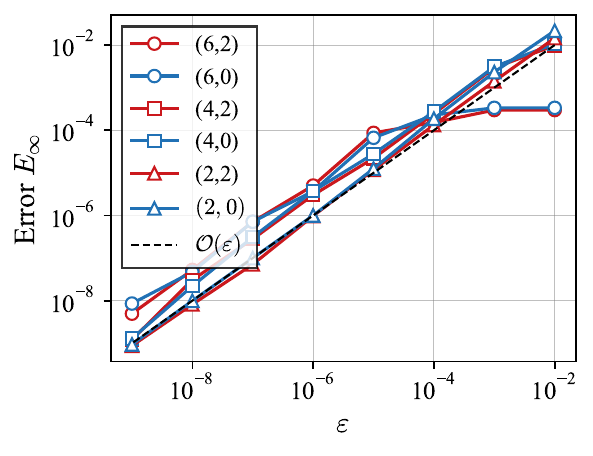}}%
    \subcaption{$L_\infty$ errors vs threshold $\varepsilon$. \label{fig:epsvserr}}
  \end{subfigure}
  \begin{subfigure}[c]{0.31\textwidth}
    \centering{\includegraphics[width=0.99\textwidth]{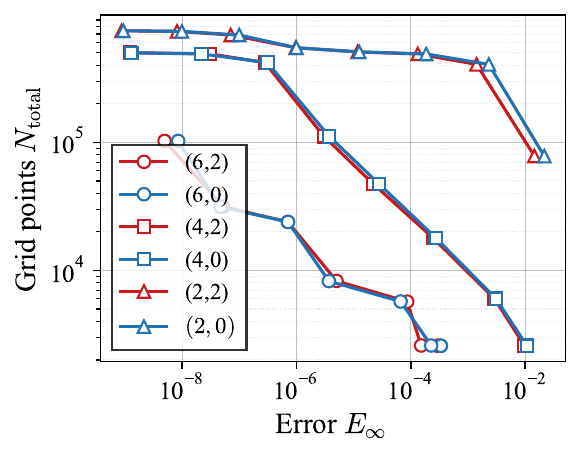}}%
    \subcaption{Compression vs error. \label{fig:errvsNp}}%
  \end{subfigure}
  \begin{subfigure}[c]{0.35\textwidth}
    \centering{\includegraphics[width=0.99\textwidth]{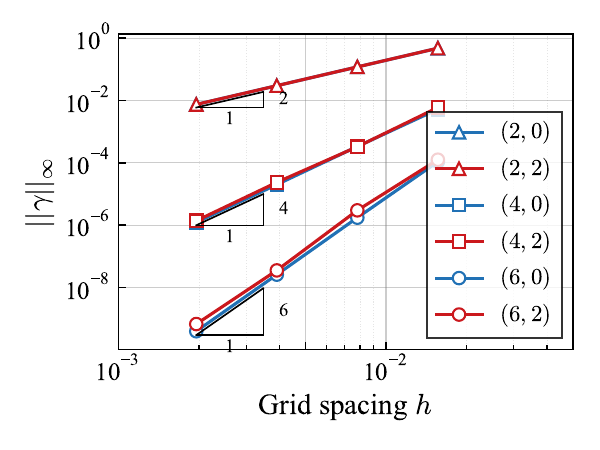}}%
    \subcaption{$||\gamma||_\infty$ vs grid spacing $h$. \label{fig:hvsmaxd}}%
  \end{subfigure}
  \caption{Errors, compression rates, and detail coefficient scaling for a static compression test of a smooth field with an immersed non-convex geometry, using lifted ($N.2$) and non-lifted ($N.0$) wavelets of varying order $N$.} \label{fig:eps_test}
\end{figure}

Figure \ref{fig:epsvserr} shows $E_\infty$ as a function of the compression threshold $\varepsilon$ for interpolating wavelets of order $(N, \tilde{N})$, with $N\in\{2,4,6\}$ and $\tilde{N}\in \{0,2\}$. The results confirm a linear relationship between $E_\infty$ and $\varepsilon$ for all tested interpolating wavelet compression on the complex domain. The constant $C$ varies for different wavelet orders, ranging from $\mathcal{O}(1)$ for the second order wavelets to $\mathcal{O}(10)$ for the sixth order wavelets. We also observe plateaus when $\varepsilon$ is relatively large, especially for high order wavelets, because the coarsest level is fixed. Figure \ref{fig:errvsNp} plots the number of active (non-zero) grid points against $E_\infty$, to reveal the compression rate of the wavelet transforms. The compression rate is highly dependent on the order of the wavelet, with sixth order wavelets being the most efficient in terms of compressing the given field $f$. Further, lifted wavelets ($\tilde{N}=2$) generally perform slightly better than their non-lifted counterparts. In Figure \ref{fig:hvsmaxd}, the maximum detail coefficient at each resolution level $||\gamma^{L_i}||_\infty$ is plotted as a function of the corresponding grid spacing $h=2^{-L_i}$. The magnitude of the detail coefficients scales as $\mathcal{O}(h^N)$ for all the interpolating wavelets tested, consistent with the theoretical analysis in Section \ref{sec:detailanal}. Overall, the results of the compression test are consistent with those obtained without immersed boundaries in \cite{gillis2022murphy}, demonstrating the effectiveness of our algorithm to handle wavelet transforms in complex domains.

Finally, Figure \ref{fig:sdview} shows the scaling and detail coefficients at the coarsest resolution level with grid spacing $h=1/64$ for the $(6,2)$ wavelet. In line with the analysis in Section \ref{sec:detailwall}, the detail coefficients near the boundary are slightly larger in magnitude compared to their free space counterparts. Nevertheless, all the first degree detail coefficients $\gamma_x$ and $\gamma_y$ scale as $\mathcal{O}(h^{N})$, and their magnitudes are of the same order of magnitude as the free-space detail coefficients. For $C^\infty$ smooth functions, the mixed detail coefficients $\gamma_{xy}$ in free space often decay as $\mathcal{O}(h^{2N})$, because the temporary detail coefficients obtained in the $x$-directional wavelet transform are still smooth across the domain, and they already have a $\mathcal{O}(h^N)$ magnitude. Near the boundary, however, this is not the case anymore because the temporary detail coefficients obtained through $x$-directional wavelet transform are discontinuous near the boundary, albeit a $\mathcal{O}(h^N)$ magnitude. As a consequence, $\gamma_{xy}$ near the boundary only decays as $\mathcal{O}(h^N)$, which is still consistent with the overall scaling $|\gamma| \sim \mathcal{O}(h^N)$.

\begin{figure}
  \centering
  \includegraphics[width=0.8\linewidth]{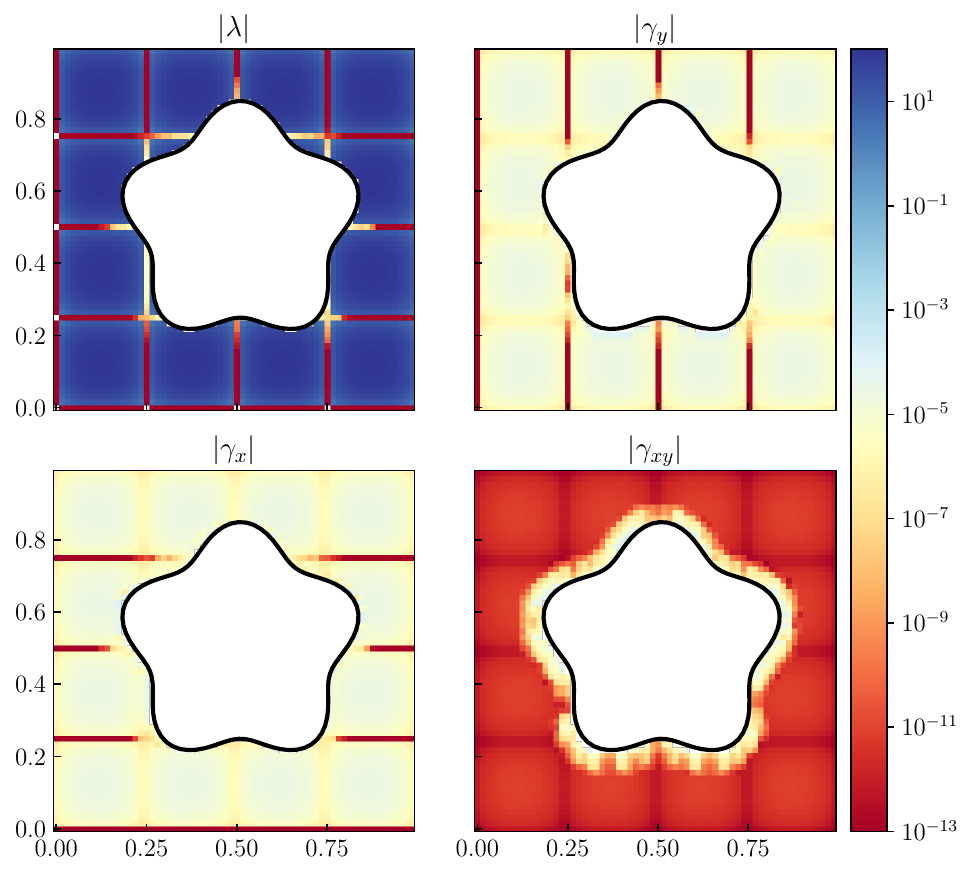}
  \caption{A (6, 2) order wavelet transform of a static field around a star shaped embedded geometry. The scaling and detail coefficients are plotted in log scale. The detail coefficients $|\gamma_x|$ and $|\gamma_y|$ have slightly larger magnitude near the star surface, due to the boundary amplification effect. $|\gamma_{xy}|$ decays as $\mathcal{O}(h^{2N})$ in free space, and $\mathcal{O}(h^N)$ near the boundary.}
  \label{fig:sdview}
\end{figure}

\subsection{Diffusion problem with Dirichlet boundary condition}\label{sec:parabolic}
In this test case, the wavelet transform algorithm is coupled with the immersed interface method \cite{gabbard2023high, gabbard2024high} using the strategy outlined in Section \ref{sec:coupling} to solve the heat equation
\begin{equation}
  u_t = \Delta u\quad \text{in }\Omega.
\end{equation} 
Here the diffusion problem is defined in a periodic unit box masked out by a star, with the following Dirichlet boundary condition:
\begin{equation}
  u|_{\partial \Omega}=\sin(5\theta(\mathbf{x}))\chi(5t)
\end{equation} 
where
\begin{equation}
  \theta(\mathbf{x}) = \arctan\left(\frac{y-y_0}{x-x_0}\right),
\end{equation}
and $\chi(t)$ is a smooth time-varying function
\begin{equation}
  \chi(t) =
  \begin{cases}
    0 \quad                                            & \text{if }t \leq 0 \\
    1 \quad                                            & \text{if }t \geq 1 \\
    \dfrac{\exp(-1/t)}{\exp(-1/t)+\exp(-1/(1-t))}\quad & \text{otherwise.}
  \end{cases}
\end{equation}

The reference solution at $t=5$, obtained through a $1024\times 1024$ simulation with an $M=4$th order finite difference spatial discretization and a third order low-storage Runge-Kutta scheme \cite{kennedy2000low}, is illustrated in Figure \ref{fig:dir_star} (left). Fixing the wavelet order at $(6, 2)$ and the ratio of the two thresholds to be $\varepsilon_r=100\varepsilon_c$, the refinement threshold is varied from $1.4\times 10^{-3}$ to $1.0 \times 10^{-7}$. We use the level-dependent adaptation threshold with $k=2$, and adapt the field every 10~steps. The pointwise error at $t=5.0$ is shown in \ref{fig:dir_star} (right). In addition, the $L_2$ and $L_\infty$ error of the numerical solutions are plotted against the refinement threshold $\varepsilon_r$ in Figure \ref{fig:dir_data1}, and against the maximum detail coefficient in Figure \ref{fig:dir_data2}. Both the $L_2$ and $L_\infty$ norm errors exhibit a linear relationship to the manually chosen error threshold, confirming  that the immersed wavelet collocation method restricts the numerical error within the manually set thresholds, up to a constant determined by time and the exact solution. Finally, the evolution of maximum detail coefficients is plotted for each simulation in Figure \ref{fig:dir_data3}. The magnitude of the detail coefficients is well kept within the dynamic region determined by the resolution level and the coarsening/refinement thresholds.

\begin{figure}
  \centering
  \includegraphics[width=1\linewidth]{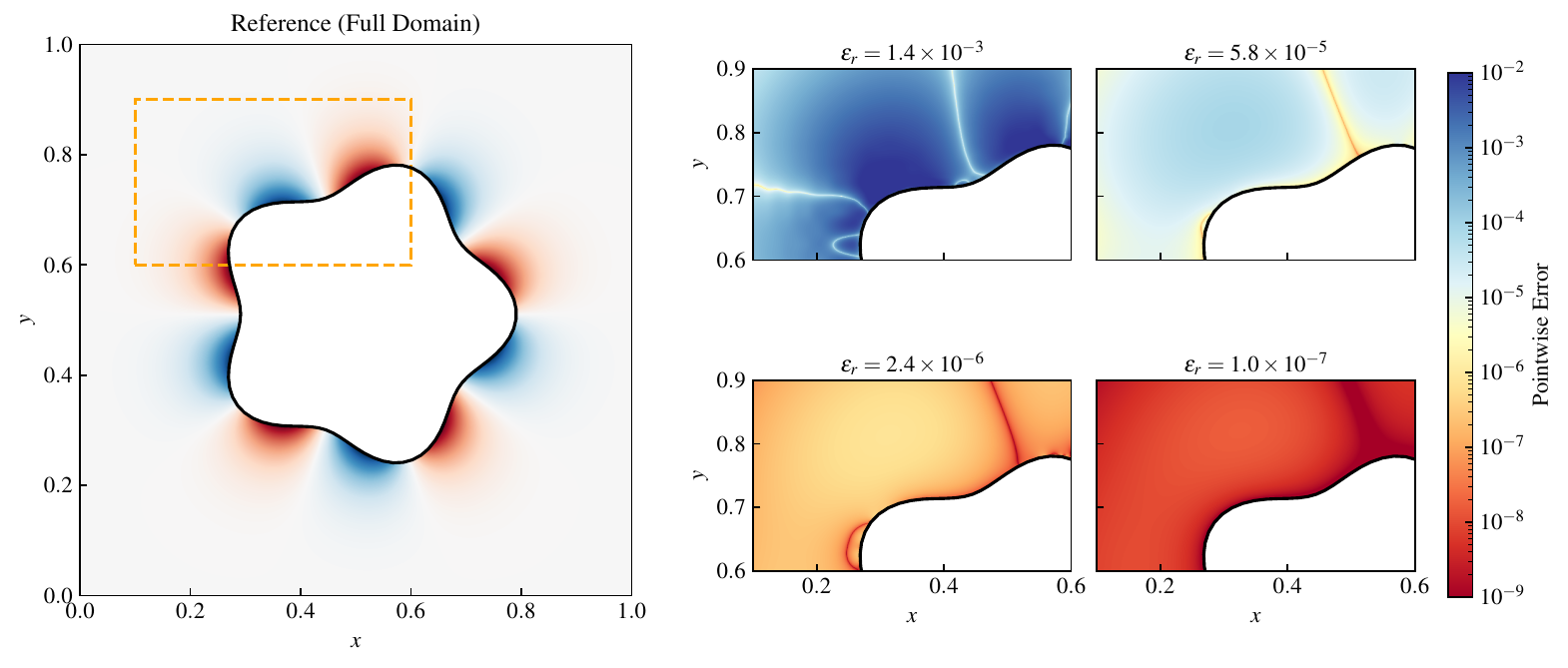}
  \caption{Left: Reference solution to the diffusion problem of Section~5.2 at t=5.0; Right: Pointwise numerical error around the top-left of the geometry as a function of refinement threshold $\varepsilon_r$. \label{fig:dir_star}}
\end{figure}

\begin{figure}[t]
  \begin{subfigure}[c]{0.32\textwidth}
    \centering{\includegraphics[width=0.99\textwidth]{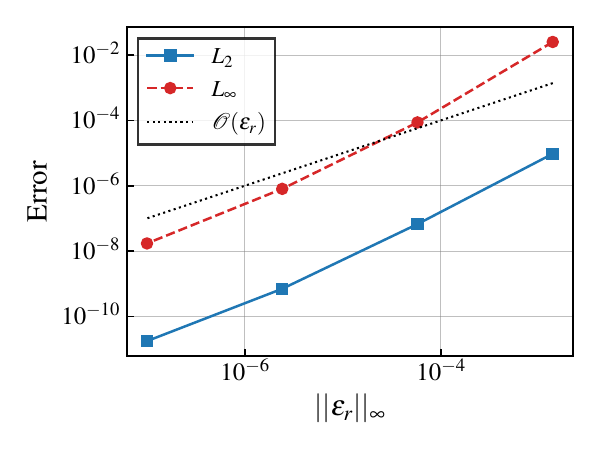}}%
    \subcaption{Relative error vs $\varepsilon_r$. \label{fig:dir_data1}}
  \end{subfigure}
  \begin{subfigure}[c]{0.32\textwidth}
    \centering{\includegraphics[width=0.99\textwidth]{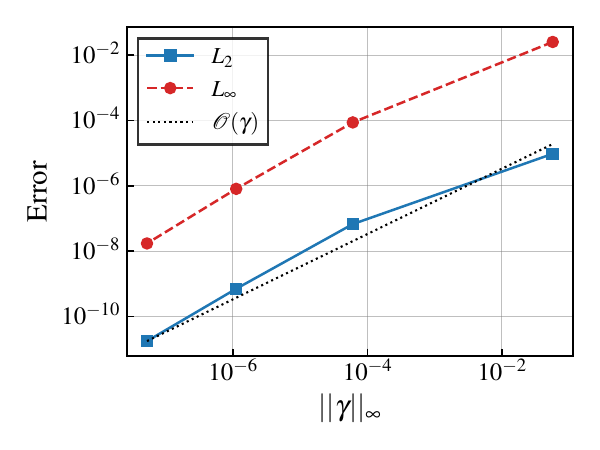}}%
    \subcaption{Relative error vs $||\gamma||_\infty$.\label{fig:dir_data2}}%
  \end{subfigure}
  \begin{subfigure}[c]{0.32\textwidth}
    \centering{\includegraphics[width=0.99\textwidth]{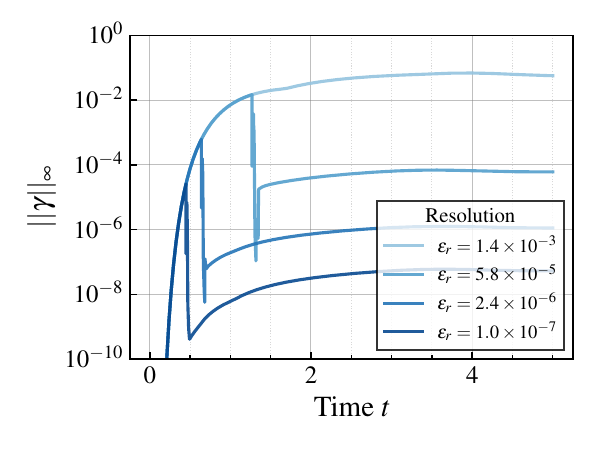}}%
    \subcaption{$||\gamma||_\infty$ over time. \label{fig:dir_data3}}%
  \end{subfigure}
  \caption{Performance of the wavelet-based grid adaptation for the diffusion problem outside of an immersed star with time-varying Dirichlet boundary conditions.}
  \label{fig:dir_data}
\end{figure}

\subsection{Navier Stokes equations}\label{sec:NS}
In this section we incorporate the proposed sharp immersed wavelet-collocation scheme into the IIM Navier-Stokes solver to simulate fluid flow with immersed boundaries. The incompressible Navier-Stokes equation takes the form
\begin{align}
  \frac{\partial \mathbf{u}}{\partial t} + \mathbf{u}\cdot \nabla \mathbf{u} & = -\frac{1}{\rho} \nabla p + \nu \Delta \mathbf{u} \\
  \nabla \cdot \mathbf{u}                                                    & = 0
\end{align} 
where $\mathbf{u}(\mathbf{x})$ is the velocity, $\rho$ is the fluid density, and $\nu$ is the kinematic viscosity. Throughout our following numerical test cases, for simplicity $\rho=1$. The wavelet-based adaptation strategy is coupled with an $M=4$th order finite difference IIM Navier-Stokes solver developed in \cite{JI2026114806}. The solver uses a projection method where the boundary condition for the pressure Poisson equation is modified so that the pseudo-pressure at the first stage of each Runge-Kutta step is a high order approximation to the true pressure. The convection term is discretized using a fifth order upwind stencil, while all the other differential operators are discretized with fourth order central difference stencils. The time step size is automatically adjusted to the maximum allowed given an imposed CFL number $\frac{\|U\|_{\infty}\Delta t}{\Delta x}$ and an imposed Fourier number $\frac{\nu \Delta t}{\Delta x^2}$; in practice the former dominates for the test cases below.

\subsubsection{A translating/rotating star immersed in incompressible fluid}
In this section the wavelet-based adaptive scheme is applied to a moving boundary problem where a translating and rotating star is immersed in a 2D incompressible fluid. The rigid body motion of the star is prescribed by its translation and rotation over time. Both the angle $\theta(t)$ and the translation in the x-direction $x_c(t)$ are chosen as translated and dilated versions of the periodic mollifier function $\chi(x)$, defined above. The y-coordinate of the center is set to be a constant $y_c=0.51$.
\begin{figure}
\centering
\includegraphics[width=1.0\linewidth]{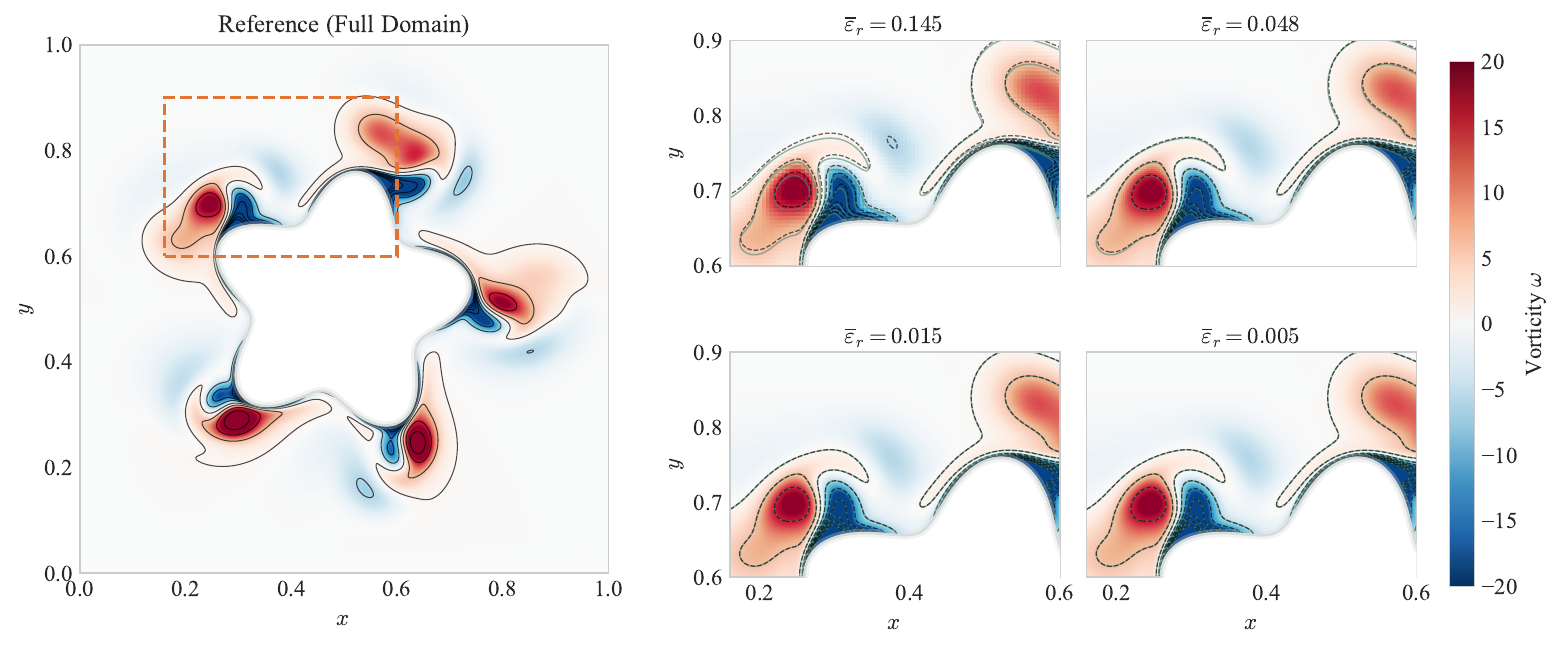}
\caption{Results of an immersed star translating and rotating in a quiescent fluid. Left: Reference solution at time $t=1.75$. Right: Vorticity fields of the adaptive simulations. The black dashed contours are adaptive solutions, while the green solid contours are the reference solution obtained using a high resolution grid. }
\label{fig:rot_star_error}
\end{figure}
The IIM finite difference discretization of the Navier-Stokes equations is set up exactly the same as in Section \ref{sec:VD}. The Reynolds number is $Re=U\overline{r}/\nu = 774.4$, where $U=1.76$ is the characteristic surface velocity of the star, and $\overline{r}$ is the star's mean diameter. The wavelet order in this test case is fixed as $(6,2)$. The refinement threshold (non-dimensionalized by $U$) ranges from $\overline\varepsilon_r= 0.005$ to $0.145$, and $\varepsilon_r=64\varepsilon_c$ throughout. In this problem, we set $k=1$, since Re $\gg \mathcal{O}(1)$, meaning the nonlinear advection term (which contains first order derivatives in space) dominates the flow evolution. We compare the numerical solution using our wavelet-based adaptive strategy with a $768\times 768$ uniform resolution reference solution. The vorticity field of the reference solution $\omega(\mathbf{x})$ at time $t=1.75$ is plotted in Figure \ref{fig:rot_star_error} (left). We also plot the numerical solutions obtained from simulations with different refinement thresholds at time $t=1.75$ in Figure \ref{fig:rot_star_error} (right). The contour lines of the numerical vorticity field (black) are plotted along with the reference contour lines (green). In this test case, the numerical solution starts to converge even at a relatively high threshold $\overline\varepsilon_r=0.048$, with most of the vortical structures and boundary layer dynamics captured qualitatively compared to the reference solution. For a threshold $\overline\varepsilon_r\leq 0.05$, the solution fully converges to the reference solution, with no obvious deviation on the level set contours compared to the reference solution.

Figure \ref{fig:detailvstime} shows the maximum detail coefficient $||\gamma||_\infty$ over time for the four different refinement thresholds. The corresponding spatial resolution over time for each simulation is displayed in Figure \ref{fig:resvstime}. We also plot the $L_2$ and $L_\infty$ norm errors of our numerical solutions with respect to $\varepsilon_r$ in Figure \ref{fig:epsvserr2}. Both $L_2$ and $L_\infty$ norm errors decay with respect to $\varepsilon_r$ at a rate slightly higher than first order. This demonstrates that the proposed wavelet-based adaptation method can provide a robust bound to the numerical error even for complex nonlinear problems with moving domain boundaries.
\begin{figure}[t]
\begin{subfigure}[c]{0.33\textwidth}
\centering{\includegraphics[width=0.99\textwidth]{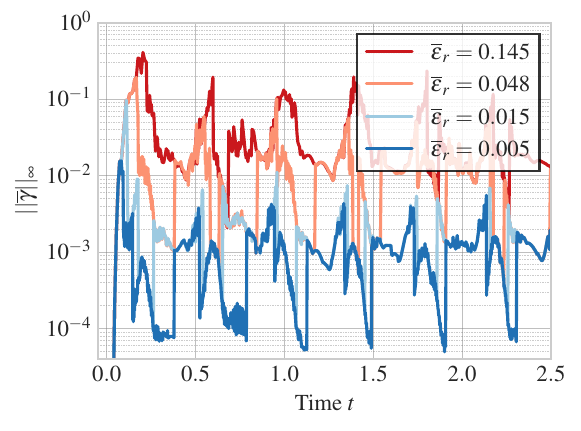}}%
\subcaption{$||\overline{\gamma}||_\infty$ vs time.\label{fig:detailvstime}}
\end{subfigure}
\begin{subfigure}[c]{0.32\textwidth}
\centering{\includegraphics[width=0.99\textwidth]{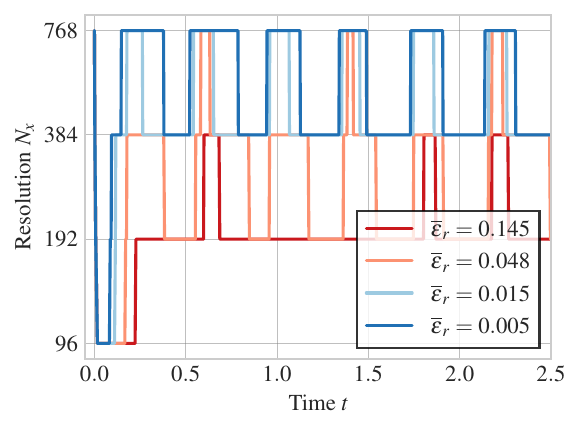}}%
\subcaption{Resolution vs time.\label{fig:resvstime}}%
\end{subfigure}
\begin{subfigure}[c]{0.32\textwidth}
\centering{\includegraphics[width=0.99\textwidth]{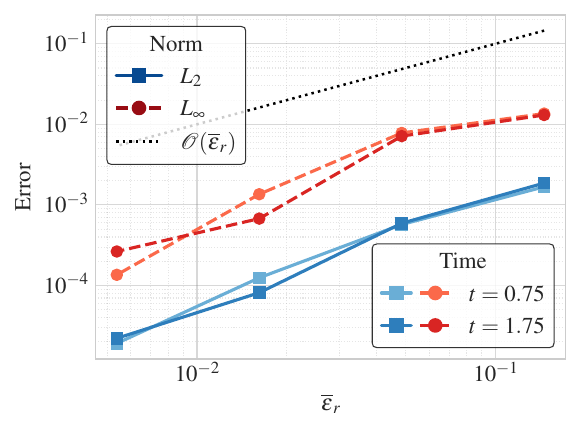}}%
\subcaption{$L_2$ and $L_\infty$ errors vs $\overline\varepsilon_r$. \label{fig:epsvserr2}}%
\end{subfigure}
\caption{Left: Maximum detail coefficients (non-dimensionalized with respect to the characteristic velocity $U$) over time for the translating/rotating star simulations with different refinement thresholds $\overline{\varepsilon}_r$ (also non-dimensionalized). Middle: The resolution levels over time for different simulations. Right: $L_2$ and $L_\infty$ norm errors of the numerical solutions as a function of $\overline{\varepsilon_r}$. A clear linear dependence is observed.\label{fig:staranalysis}}
\end{figure}

\subsubsection{Vortex dipole impinging on an immersed wall} \label{sec:VD}
In this final section, we apply our method to a two-dimensional dipole-wall collision benchmark test case, which was first introduced in \cite{clercx2006normal}. In this test, a thin boundary layer is generated over time near the immersed wall, which makes it a challenging simulation to test the convergence and accuracy of numerical solvers. In \cite{keetels2007fourier}, several numerical methods have been applied to solve this dipole-wall collision test case, including a penalization-based immersed method, spectral methods, and wavelet-Galerkin methods. Recently, in \cite{JI2026114806}, a fourth order immersed interface scheme is used to solve the same dipole-wall problem within an immersed domain at a fixed angle with respect to the Cartesian grid.

Here, we follow the exact same setup of the problem as the one in \cite{JI2026114806}. The Reynolds number is fixed at $Re=1000$, and the fluid domain $\Omega$ is defined as an immersed rounded square rotated at an angle $3\pi/10$ with respect to the computational domain $C=[0,1]^2$, which is discretized with a uniform Cartesian grid. The square has a side length $L=0.65123$ and is centered at the point $(0.501, 0.499)$.  To properly resolve the fluid domain, the corner of $\Omega$ is rounded with curvature $L/10$.

The initial resolution for all the adaptive simulations is $N_x=384$. The results are compared to a fixed-resolution reference solution obtained with $N_x=1536$. The CFL number is fixed to be $0.4$ throughout the simulations. A $(6, 2)$-order interpolating wavelet is used to perform grid adaptation every ten time steps. We vary the refinement threshold $\overline\varepsilon_r$ from \numrange{0.223}{0.0017}, and fix the ratio $\overline\varepsilon_r / \overline\varepsilon_c = 32$. Here the non-dimensional thresholds are defined as $\overline\varepsilon_{r,c}=\varepsilon_{r,c}/U$ where $U=\max|\mathbf{u}_0(\mathbf{x})|$ is the characteristic velocity magnitude. In this specific test case, the level-independent adaptation strategy corresponding to $k=0$, as described in section~\ref{sec:coupling}, is used. There are two reasons for this choice: first,  the vortex-wall collision causes strong nonlinearity, so that the error bound of Proposition~\ref{prop:3} does not hold; second, the detail coefficients are in the pre-asymptotic regime for relatively low resolutions ($<768$). In the pre-asymptotic regime, a level-independent strategy is stricter than the level-dependent one, and the situation reverses in the asymptotic regime.

\begin{figure}
\centering
\includegraphics[width=1.0\linewidth]{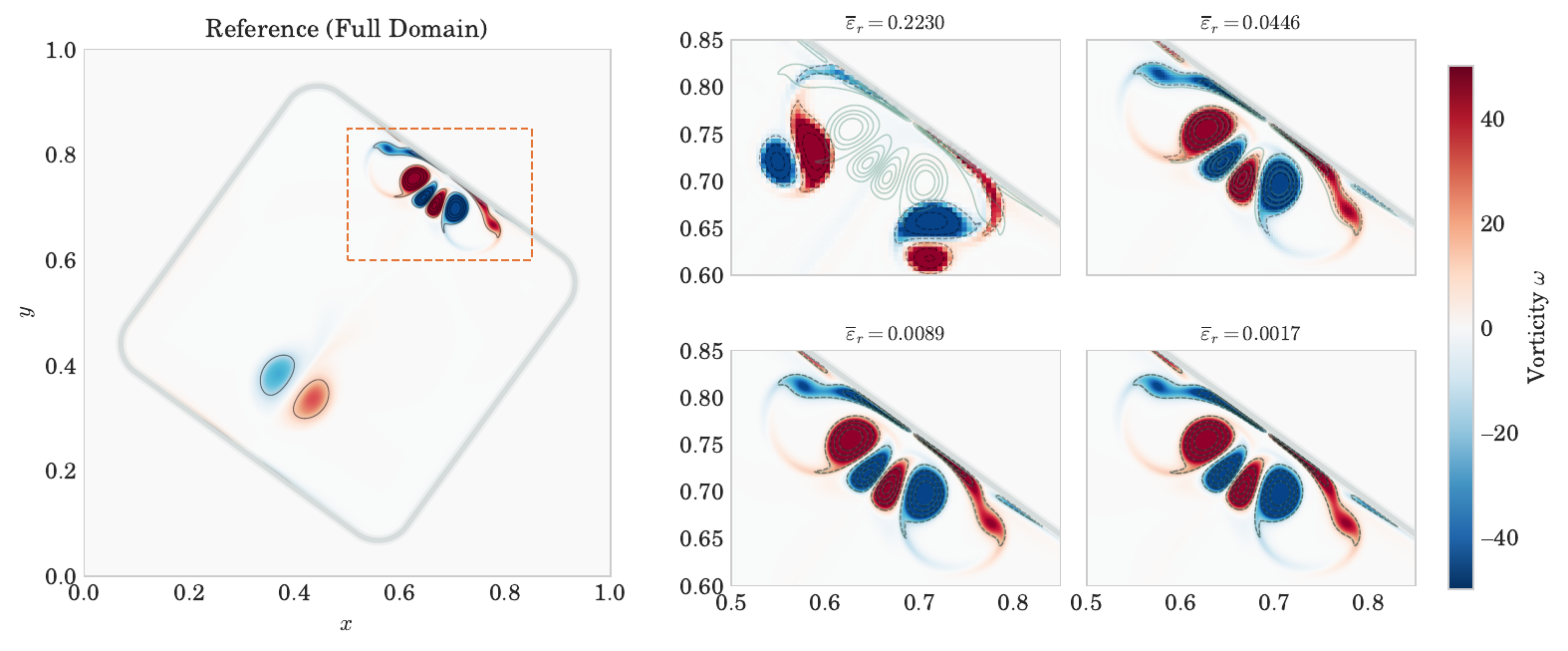}
\caption{Flow field of a vortex dipole-wall collision inside an immersed domain rotated with respect to the Cartesian background grid. Left: Vorticity (both field and contour) at time $t=0.6$ for the reference solution at $1536^2$ resolution. Right: Adaptive simulations with different thresholds $\varepsilon_c$ at the same time. The vorticity contours of the adaptive simulations are black, while the contour of the reference solution is green.} \label{fig:dipole_vort}
\end{figure}

Figure \ref{fig:dipole_vort} compares the vorticity contours of adaptive simulations with different coarsening / refinement thresholds to the reference solution obtained with uniform $1536^2$ resolution. For $\overline\varepsilon_r=0.223$ the solution diverges completely from the reference solution, as the simulation is allowed to produce an error of magnitude $\mathcal{O}(1)$. For smaller $\overline\varepsilon_r$ the solution gradually converges to the reference solution as the threshold decreases in magnitude. When $\overline\varepsilon_r=\num{0.0017}$ the adaptive solution fully converges to the reference solution, with no obvious deviation. Figure \ref{fig:dipole_gamma-time} plots the maximum detail coefficient $||\gamma||_\infty$ as a function of time for different adaptive simulations. The results show that the magnitude of the detail coefficient is kept well within the interval $[\overline\varepsilon_c, \overline\varepsilon_r]$. The spatial resolution over time is shown in Figure \ref{fig:dipole_res-time}, demonstrating the coarsening and refinement of grid as wall-collision events unfold.

\begin{figure}[htbp]
\centering

\begin{subfigure}[t]{0.48\textwidth}
\centering
\includegraphics[width=\linewidth]{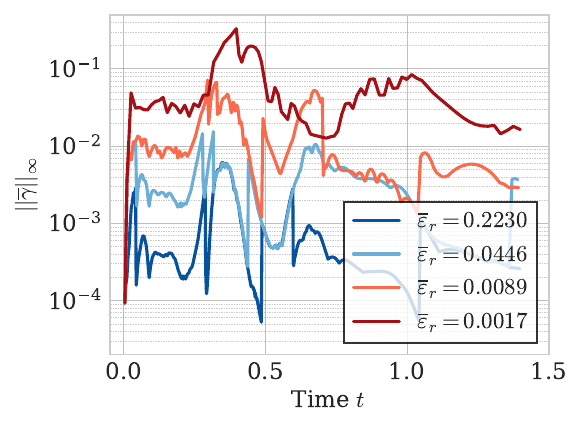}
\caption{$||\gamma||_\infty$ vs.\ time}
\label{fig:dipole_gamma-time}
\end{subfigure}\hfill
\begin{subfigure}[t]{0.48\textwidth}
\centering
\includegraphics[width=\linewidth]{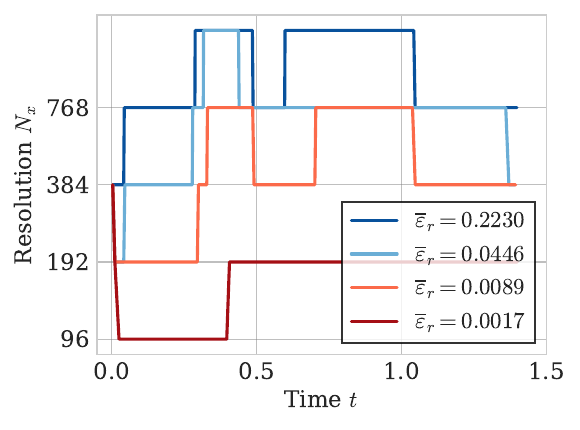}
\caption{Resolution vs. time}
\label{fig:dipole_res-time}
\end{subfigure}

\caption{Temporal variation of the maximum detail coefficient (left) and the grid resolution (right) for for different refinement thresholds $\varepsilon_r$ of the vortex dipole simulations.}
\label{fig:dipole_res}
\end{figure}

\begin{figure}[htbp]
\centering

\begin{subfigure}[t]{0.48\textwidth}
\centering
\includegraphics[width=\linewidth]{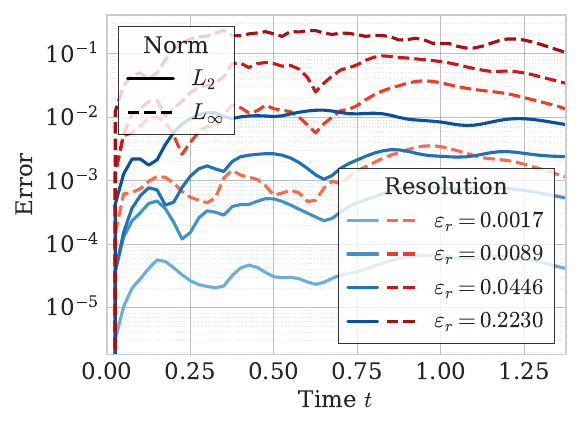}
\caption{Error vs.\ time}
\label{fig:dipole_error-time}
\end{subfigure}\hfill
\begin{subfigure}[t]{0.48\textwidth}
\centering
\includegraphics[width=\linewidth]{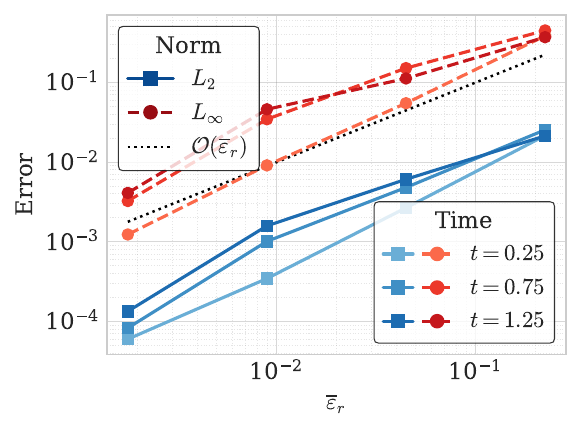}
\caption{Error vs.\ $\varepsilon_r$}
\label{fig:dipole_error-eps}
\end{subfigure}

\caption{Left: $L_2$ and $L_\infty$ norm errors over time for different vortex dipole simulations. Right: Error convergence with respect to $\varepsilon_r$ at different time slices. The $L_2$ norm errors show a linear dependence with respect to $\varepsilon_r$, while $L_\infty$ is suboptimal. }
\label{fig:dipole_error}
\end{figure}

Figure \ref{fig:dipole_error-time} shows the $L_2$- and $L_\infty$-norm errors of the $x$-component velocity field over time for simulations with different $\overline\varepsilon_r$. In Figure \ref{fig:dipole_error-eps}, errors as functions of the refinement threshold $\varepsilon_r$ are also plotted for time slices $t=0.25, 0.75, 1.25$. The $L_2$ error is well bounded by the chosen refinement threshold $\overline\varepsilon_r$, showing a clear linear dependence; the $L_\infty$ error is bounded by the threshold with sub-optimality due to the strong nonlinearity that occurred in the simulation. Nevertheless, tightening the refinement threshold still corresponds to a significant decrease in the $L_\infty$ error.
\begin{figure}[htbp]
\centering

\begin{subfigure}[t]{0.48\textwidth}
\centering
\includegraphics[width=\linewidth]{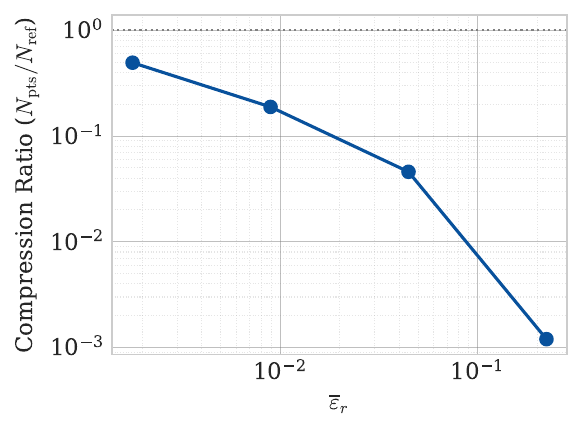}
\caption{Compression rate vs $\overline\varepsilon_r$.}
\label{fig:comp-eps}
\end{subfigure}\hfill
\begin{subfigure}[t]{0.48\textwidth}
\centering
\includegraphics[width=\linewidth]{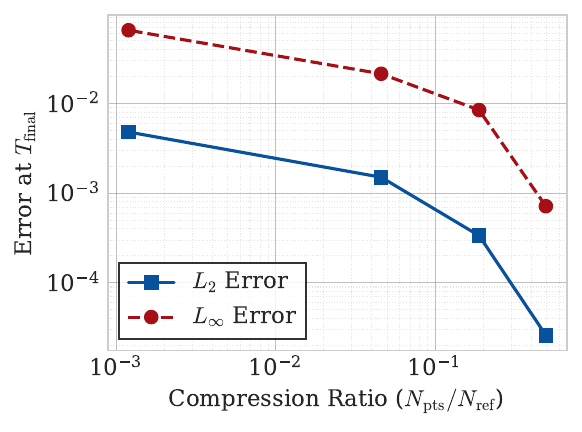}
\caption{Relative error vs compression rate.}
\label{fig:err-comp}
\end{subfigure}

\caption{Left: Compression rate for different adaptive solutions as a function of the refinement threshold $\overline\varepsilon_r$ for the vortex dipole-wall interaction. Right: The $L_2$ and $L_\infty$ errors as functions of the compression rate.}
\label{fig:compression}
\end{figure}

Finally, we investigate the relationship between the refinement threshold, the relative error, and the total compression rate across the entire simulation. The total compression rate is computed as the ratio between the number of spatiotemporal degrees-of-freedom (DOFs) in the adaptive simulations and the number of DOFs in the $1536^2$ reference solution. In Figure \ref{fig:comp-eps}, the compression rate is plotted against the refinement threshold $\varepsilon_r$. The $L_2$ and $L_\infty$ norm errors as functions of the compression rate are also shown in the same Figure. Overall, a larger numerical error is obtained when $\overline\varepsilon_r$ is larger, but this is compensated by the gain in a better compression rate, which yields a better computational efficiency. 

\section{Conclusion}\label{sec:conclusion}
This work presents a novel wavelet-based temporally-adaptive grid method designed for sharp immersed discretization methods. The proposed method provides a robust dependency of the numerical error on the user-specified refinement threshold, detecting grid adaptation needs everywhere, including around the immersed geometry. To achieve this, a novel interpolating wavelet transform algorithm is constructed on arbitrary irregular domains, preserving the order $N$ of the wavelet everywhere. This remains true even for irregular domains with non-convex features. Like the original interpolating wavelet transform \cite{deslauriers_interpolation_1987, donoho1992interpolating, donoho1999deslauriers}, the modified version has an $\mathcal{O}(N_p)$ complexity for both forward and inverse wavelet transforms, where $N_p$ is the number of grid points. Unlike other existing wavelet algorithms such as \cite{li2000shape, minami20013}, the proposed approach takes advantage of proper polynomial extrapolations so that the detail coefficients decay as $\mathcal{O}(h^N)$.

We combined the proposed wavelet transform algorithm with an immersed interface solver \cite{JI2026114806} to achieve a temporally adaptive grid resolution strategy. Using a mathematical proof, it is shown that for linear PDEs the user-defined refinement threshold for the grid adaptation strategy provides an upper bound to the numerical error of the PDE solution over time. This guarantees that the wavelet collocation adaptive resolution strategy robustly keeps numerical errors in the PDE solution within the manually-chosen tolerance (up to a constant). 
The error scaling is validated via a linear diffusion problem around a star-shaped immersed body. For more practical scenarios, we also applied the method to a vortex dipole-wall collision problem which incurs strong nonlinear dynamics, and a translating/rotating immersed star problem which incurs a deformed domain. From all these tests, we observe that the immersed wavelet collocation method suppresses the numerical error under the manually set threshold $\varepsilon_r$.

Extension of the temporal grid adaptation to three-dimensional domains is straightforward: in free space and near a convex boundary, an $xyz$ dimension-split approach to the wavelet transform can be adopted. Near a boundary with concave geometry, a similar least squares polynomial extrapolation can be applied, with the only tweak that the least squares region in 3D becomes a half ellipsoid, as in previous extensions of the IIM and wavelet transform to 3D \cite{gillis2022murphy,gabbard2024high}.

A more challenging future direction is the extension to spatially adaptive grid resolutions. Performing a wavelet transform near an immersed interface or boundary with a spatial resolution jump poses some unique challenges. The elliptical region proposed here to estimate boundary values and derivatives would generally cover regions of varying spatial resolution. Ideally, a local high-order wavelet transform would refine all these regions locally to a constant resolution before evaluating the least-squares polynomial fit; however, this wavelet transform potentially requires again the boundary values and derivatives that we are aiming to reconstruct in the first place. This interdependency leads to a series of local linear systems of equations that needs to be solved around the interface; formulating and solving those systems efficiently is left for future work.

\section*{Declarations}
The authors declare that they have no known competing financial interests or personal relationships that could have appeared to influence the work reported in this paper.

\section*{Data availability}
The datasets generated in this study are available from the corresponding author on reasonable request.

\begin{appendices}
\section{Analysis of detail coefficients}
\label{sec:detailanal}
In this section, we analyze the magnitude of the detail coefficients with respect to the grid size $h$, both in free space and near the boundary. We then explain why using boundary conditions in our interpolating wavelet transform is helpful in reducing the magnitude of the detail coefficients near the wall. Finally, we show that our treatment of the concave geometries preserves the formal order of the interpolating wavelet transform.

\subsection{Sufficiently large intervals} \label{sec:detailwall}
The calculation of the detail coefficients in the interpolating wavelet transforms can be viewed as a polynomial-fitting problem. Here, for the sake of simplicity, we consider a one-dimensional grid $\mathcal{X}^{L+1} = \left\{x_j^{L+1} \equiv jh, j\in \mathbb{Z}\right\}$ where $h=2^{-(L+1)}$. Again we assume that $\lambda^{L+1}_j=f(x_j^{L+1})$ for some function $f\in \mathcal{C}^{\infty}$. Without loss of generality, we consider the detail coefficient $\gamma^L_{0}$ in the free space, which can be expanded as
\begin{align}
\begin{split}\gamma^L_{0} & = \lambda_{1}^{L+1} - \sum_{j=-N/2+1}^{N/2}\tilde{S}_{0,j} \lambda_{2j}^{L+1} \\
  & = f(x_j^{L+1}) - (\tilde{S} f) (x_j^{L+1})
\end{split}
\end{align}
where $(\tilde{S} f)$ is the unique $(N-1)$-th degree polynomial interpolated from the set of data $\left\{f\left(x^{L+1}_{2j}\right), j\in [\![-N/2+1, N/2]\!]\right\}$. The Lagrange remainder theorem \cite{suli2003introduction} then yields an upper bound for the detail coefficient:
\begin{align}
\begin{split}
  \left|\gamma^L_0\right| & = \left|f(x_j^{L+1}) - (\tilde{S} f) (x_j^{L+1})\right|                                           \\
  & = \left|\frac{f^{(N)}(\zeta)}{N!}\prod_{j=-N/2+1}^{N/2}\left(x_1^{L+1}-x_{2j}^{L+1}\right)\right| \\
  & \leq \left|\frac{M_1}{N!}\prod_{j=-N/2+1}^{N/2}(1-2j)\right| h^N=\Gamma_1,
\end{split}
\end{align} where $M_1=\max_{x_{-N+2}^{L+1}\leq \zeta\leq x_{N}^{L+1}}\left|f^{(N)}(\zeta)\right|$. From here we can see that the detail coefficient $\gamma^L_0$ indeed scales as $\mathcal{O}(h^N)$.

Now let us further consider the case where we have a boundary at $x_b=(1-\psi)h, \psi\in(0, 1]$, and thus the one-dimensional grid becomes $\mathcal{X}^{L+1'}=\left\{x_j^{L+1}, j\in \mathbb{N}^+\right\}$. By the uniqueness of $(N-1)$-th degree polynomials, the extrapolation-based wavelet transform specified in Section \ref{sec:wt1d} can be reinterpreted as
\begin{equation}
\gamma_0^{L'} = f(x_j^{L+1}) - (\tilde{S}' f) (x_j^{L+1})
\end{equation} where $\tilde{S}' f$ is the unique $(N-1)$-th degree polynomial interpolated from a set of one-sided data $\left\{f\left(x_{2j}^{L+1}\right), j\in [\![1, N]\!]\right\}$. Again, applying Lagrange remainder theorem gives us an upper bound for this freshly computed detail coefficient $\gamma_0^{L'}$:
\begin{align}
\left|\gamma^{L'}_0\right| \leq \left|\frac{M_2}{N!}\prod_{j=1}^{N}(1-2j)\right| h^N = \Gamma_2
\end{align} where $M_2=\max_{x_{1}^{L+1}\leq \zeta\leq x_{2N}^{L+1}}\left|f^{(N)}(\zeta)\right|$. We assume that $f^{(N)}(x)$ is Lipschitz continuous so that $M_2/M_1 \rightarrow 1$ as $h \rightarrow 0$. With this assumption, we compare the upper bounds for $\gamma_0^L$ and $\gamma_0^{L'}$ as $h \rightarrow 0$:
\begin{align}
\lim_{h \rightarrow 0} \Gamma_2 / \Gamma_1
& =\left.\left|\prod_{j=1}^{N}(1-2j)\right|\right/\left|\prod_{j=-N/2+1}^{N/2}(1-2j)\right|.
\end{align} 
This ratio approximates the Lebesgue constant for wavelet interpolation near the boundary, and we denote it as $\varrho=\lim_{h\rightarrow 0} \Gamma_2/\Gamma_1$. Notice that it is purely a function of the wavelet order $N$. Figure \ref{fig:bmf} shows $\varrho(N)$ for $N=2,4,6$ . Notice that this function grows exponentially, \textit{i.e.}, $\varrho(N) \sim \mathcal{O}(2^N)$. This means that for very high order wavelets, the detail coefficients near the boundary are likely to be magnified significantly.
\begin{figure}
\centering
\includegraphics[width=0.6\linewidth]{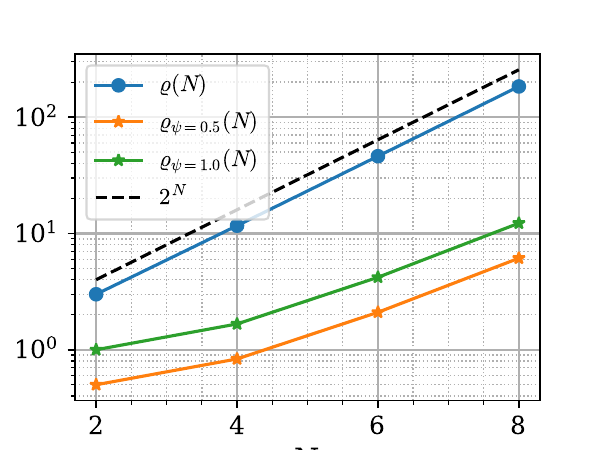}
\caption{Lebesgue constants near the boundary for different wavelet orders. Blue: boundary condition not used; Orange: boundary condition is used, and $\psi=0.5$; Green: boundary condition is used, and $\psi=1.0$.}
\label{fig:bmf}
\end{figure}

We can also investigate the effect of incorporating the boundary condition into our wavelet transform algorithm. Using the boundary condition $f(x_c)$ where $x_c=(1-\psi)h$ when performing the wavelet transform, we effectively constructed another polynomial $p''$ from the set of data $\left\{f\left(x_c\right)\right\}\cup \left\{f\left(x_{2j}^{L+1}\right), j\in[\![1, N-1]\!]\right\}$. The detail coefficient in this context is defined as
\begin{equation}
\gamma_0^{L''} = f(x_j^{L+1}) - (\tilde{S}'' f) (x_j^{L+1}).
\end{equation} We can obtain its upper bound by the following computation, which once again uses the Lagrange remainder theorem:
\begin{align}
\left|\gamma^{L''}_0\right| \leq \left|\frac{M_3}{N!}\psi\prod_{j=1}^{N-1}(1-2j)\right| h^N = \Gamma_3
\end{align}
where $M_3=\max_{x_b\leq \zeta\leq x_{2N}^{L+1}}\left|f^{(N)}(\zeta)\right|$. We compare this upper bound for $\gamma_0^{L''}$ with the upper bound for $\gamma_0^L$ as $h \rightarrow 0$:
\begin{align}
\lim_{h \rightarrow 0} \Gamma_3 / \Gamma_1 =\left.\left|\psi\prod_{j=1}^{N-1}(1-2j)\right|\right/\left|\prod_{j=-N/2+1}^{N/2}(1-2j)\right|.
\end{align} Again, we assumed Lipschitz continuity of the function $f^{(N)}(x)$, so that $\lim_{h\rightarrow 0}M_3/M_1=1$. It is easy to see that $\varrho = \lim_{h\rightarrow0}\Gamma_3/\Gamma_1$ is a function of both $N$ and $\psi$. Figure \ref{fig:bmf} shows both $\varrho(N, \psi=0.5)$ (the average case) and $\varrho(N, \psi=1.0)$ (the worst case). It is obvious from the figure that using a boundary condition in the wavelet transform drastically mitigates the Runge phenomenon under mild $N$. In practice, the wavelets used are often of order up to six, a regime where the Lebesgue constant is manageable.

\subsection{Narrow intervals} \label{sec:hermite}
As discussed in Section \ref{sec:concave}, when dealing with concave geometries, we proposed to fit a Hermite-like 1D interpolant $q(x)$ to extrapolate the ghost point values. This polynomial $q(x)$ differs slightly from the traditional definition of a Hermite interpolant. A general Hermite interpolant can be formulated as follows \cite{han2009theoretical}. Find $p(x)\in \mathcal{P}_{K-1}, K=\sum_{i=1}^n(m_i+1)$ that satisfies the conditions
\begin{equation}
p^{(j)}(x_i)=f^{(j)}(x_i), \quad 0\leq j\leq m_i, \ 1\leq i\leq n.
\end{equation} For such a polynomial $p(x)$, we have the error estimate
\begin{align}
|f(x) - p(x)|=\left|\frac{f^{(K)}(\zeta_x)}{K!}\prod_{i=1}^n (x-x_i)^{m_i+1}\right|.
\end{align}

Our polynomial interpolant $q(x)$ does not satisfy the definition of a Hermite interpolant of the original function $f(x)$, due to the following two facts. First, depending on the location of the control points $x_{c,1}$ and $x_{c,2}$, we might or might not include the conditions
\begin{equation}q(x_{c,1})=f(x_{c,1}) \quad\text{and}\quad q(x_{c,2})=f(x_{c,2}). \label{eq:hermitebc}
\end{equation}
Specifically, when $x_{c,1}$ or $x_{c,2}$ is too close to the nearby interpolation point, we will not use the corresponding condition in Equation \ref{eq:hermitebc}, and replace it with a condition that matches the higher order derivative of $f(x)$ and $q(x)$. In this case, our interpolant $q(x)$ does not satisfy the definition of a Hermite interpolant. Second, conditions involving derivative of order $j$ are only satisfied up to an error of magnitude $\mathcal{O}(h^{N-j})$ (see Equation \ref{eq:hermiteerror}), which obviously violates the definition of a Hermite interpolant as well. However, our Hermite-like interpolant $q(x)$ is still a high-order approximation to $f(x)$, and the detail coefficients obtained through $q(x)$ still scale as $\mathcal{O}(h^N)$.

To show this, we consider the following set-up. Suppose that $x_{c,1}$ and $x_{c,2}$ are two adjacent control points due to the concavity of the boundary, such that there is a set of $k<N$ available grid points $\{x_1,\cdots, x_k\}$ for extrapolation, where $x_{c,1}$ and $x_{c,2}$ may or may not be included in this set, depending on the specific location of the two control points. The distance between $x_{c,1}$ and $x_{c,2}$ is of order $\mathcal{O}(h)$. Based on this set-up, we make the following two propositions.
\begin{proposition}
Let $f(x)\in \mathcal{C}^{\infty}$. Suppose a polynomial of degree $N-1$ satisfies the condition
\begin{equation}
  \begin{cases}
    p(x_i) = f(x_i), \quad i=1,\cdots, k;                          \\
    p^{(m)}(x_{c,1}) = f^{(m)}(x_{c,1}), \quad m=1,\cdots, \alpha; \\
    p^{(n)}(x_{c,2}) = f^{(n)}(x_{c,2}), \quad n=1,\cdots, \beta
  \end{cases} \quad \text{where } \alpha+\beta+k=N. \label{eq:pdef}
\end{equation} Then for any $x\in [x_{c,1}, x_{c,2}]$, $|f(x)-p(x)|\sim \mathcal{O}(h^N)$.\label{prop:1}
\end{proposition}

\begin{proposition}
Let $\varepsilon(x)$ be a polynomial of degree $N-1$ that satisfies
\begin{equation}
  \begin{cases}
    \varepsilon(x_i) = 0, \quad i=1, \cdots, k;                                    \\
    \varepsilon^{(m)}(x_{c,1})\sim\mathcal{O}(h^{N-m}), \quad m=1, \cdots, \alpha; \\
    \varepsilon^{(n)}(x_{c,2})\sim\mathcal{O}(h^{N-n}), \quad n=1, \cdots, \beta
  \end{cases} \quad \text{where } \alpha+\beta+k=N. \label{eq:epsfun}
\end{equation} Then for any $x\in [x_{c,1}, x_{c,2}]$, $|\varepsilon(x)|\sim \mathcal{O}(h^N)$. \label{prop:2}
\end{proposition}

We can use the above two propositions \ref{prop:1} and \ref{prop:2} to show that our wavelet transform algorithm preserves the wavelet order even with the presence of concave geometries. If we define a polynomial $p(x)$ using Equation \ref{eq:pdef}, and compare it with our Hermite-like interpolant $q(x)$, it is easy to show that their difference $\varepsilon(x)=p(x)-q(x)$ satisfies Equation \ref{eq:epsfun}. We therefore conclude by the triangle inequality that for any $x\in [x_{c,1}, x_{c,2}]$,
\begin{equation}
|f(x)-q(x)| \leq |f(x)-p(x)|+|p(x)-q(x)| = |f(x)-p(x)|+|\varepsilon(x)|\sim \mathcal{O}(h^N),
\end{equation} and hence the detail coefficients computed via the polynomial $q(x)$ should scale as $\mathcal{O}(h^N)$.

\section{Proofs}
\textbf{Proof to Proposition \ref{prop:1}:} Suppose $x\in [x_{c,1}, x_{c,2}]$. The proposition is trivial for the case $x=x_i$, so we only consider the case where $x\neq x_i$. Let $E(x)=f(x)-p(x)$, where $p(x)$ is the unique polynomial interpolant through $\{x_i\}_{i=0}^k$. For a fixed $x$, we consider the function
\begin{equation}
g(t) = E(t)-\frac{\omega(t)}{\omega(x)}E(x)
\end{equation} where
\begin{equation}
\omega(x)=(x-x_{c,1})^\alpha (x-x_{c,2})^\beta \prod_{i=0}^{k}(x-x_i).
\end{equation} Notice that $g(t)$ has $k+1$ distinct roots for $t \in (x_{c,1}, x_{c,2})$, namely $t=x$ and $t=x_i$. By the mean value theorem, $g'(t)$ has $k$ distinct roots within $(x_{c,1}, x_{c,2})$. In addition, by direct substitution, it can be found that $g^{(1)}(x_{c,1})=g^{(1)}(x_{c,2})=0$, which means $x_{c, 1}$ and $x_{c, 2}$ are also roots of $g(t)$. Hence, we have $k+2$ distinct roots within $[x_{c,1}, x_{c,2}]$. We may apply the mean value theorem repetitively on $g^{(i)}$ to obtain the roots of $g^{(i+1)}$ within the interval $[x_{c,1}, x_{c,2}]$. Notice that on the boundaries, $x_{c,1}$ is a root of $g^{(i)}$ for $i<\alpha$, and $x_{c,2}$ is a root of $g^{(i)}$ for $i<\beta$. Therefore, if we apply this process for $N$ times, $(k+1+\alpha+\beta-N)=1$ root will be found within the closed interval $[x_{c,1}, x_{c,2}]$ for the function $g^{(N)}(t)$. That is, there exists a unique root $\zeta_x\in [x_{c,1}, x_{c,2}]$ that satisfies
\begin{equation}
0=g^{(N)}(\zeta_x)=f^{(N)}(\zeta_x)-\frac{N!}{\omega(x)}E(x).
\end{equation}
Hence, we obtain
\begin{equation}
E(x)=\frac{f^{(N)}(\zeta_x)}{N!} \omega(x).
\end{equation} Finally, we make note that $\omega(x)\sim \mathcal{O}(h^N)$ for $x\in [x_{c,1}, x_{c,2}]$. $\square$\\

\textbf{Proof to Proposition \ref{prop:2}:} Notice that $\varepsilon(x)$ can be written as
\begin{equation}
\varepsilon(x)=\sum_{m=1}^\alpha \eta_m \pi_{1, m} (x)+\sum_{n=1}^{\beta} \xi_n\pi_{2,n}(x)
\end{equation} where
\begin{equation}
\begin{cases}
  \eta_m = \varepsilon^{(m)}(x_{c,1}) \sim \mathcal{O}(h^{N-m}), \quad m=1, \cdots, \alpha; \\
  \xi_n = \varepsilon^{(n)}(x_{c,2}) \sim \mathcal{O}(h^{N-n}), \quad n=1, \cdots, \beta.
\end{cases}
\end{equation} Here, $\pi_{1, m}(x)$ is the unique cardinal function satisfying
\begin{equation}
\begin{cases}
  \pi_{1, m}(x_i) = 0,\quad i=1, \cdots, k;                            \\
  \pi_{1, m}^{(n)}(x_{c,1}) = \delta_{m,n}, \quad n=1, \cdots, \alpha; \\
  \pi_{1, m}^{(n)}(x_{c,2}) = 0, \quad n=1, \cdots, \beta.
\end{cases}
\end{equation} $\pi_{2, n}(x)$ is defined similarly. Note that $\pi_{1, m}(x)$ must scale as $\mathcal{O}(h^m)$. To see this, consider the unique $(N-1)$-th degree polynomial $\tilde{\pi}_{1, m}$ defined through
\begin{equation}
\begin{cases}
  \tilde{\pi}_{1, m}(x_i/h) = 0,\quad i=1, \cdots, k;                            \\
  \tilde{\pi}_{1, m}^{(n)}(x_{c,1}/h) = \delta_{m,n}, \quad n=1, \cdots, \alpha; \\
  \tilde{\pi}_{1, m}^{(n)}(x_{c,2}/h) = 0, \quad n=1, \cdots, \beta,
\end{cases}
\end{equation} so it is a rescaled version of $\pi_{1,m}(x)$ on a normalized grid with grid spacing equal to 1. Therefore, we write
\begin{equation}
\tilde{\pi}_{1,m}(x/h) = C_h \pi_{1,m}(x).
\end{equation} Differentiating $m$ times, we obtain
\begin{equation}
\tilde{\pi}_{1,m}^{(m)}(x/h)h^{-m} = C_h \pi_{1,m}^{(m)}(x).
\end{equation} But we know that $\tilde{\pi}_{1,m}^{(m)}(x_{c,1}/h)=\pi_{1,m}^{(m)}(x_{c,1})=1$.This implies that
\begin{equation}
C_h=h^{-m}
\end{equation} and thus
\begin{equation}
\frac{\pi_{1,m}(x)}{\tilde{\pi}_{1,m}(x/h)} = \frac{1}{C_h}=h^m.
\end{equation} Therefore, $\pi_{1,m}(x)\sim \mathcal{O}(h^m)$. Similar result holds for $\pi_{2,m}(x)$. We conclude with the following estimate:
\begin{align}
\varepsilon(x) & =\sum_{m=1}^\alpha \eta_m \pi_{1, m} (x)+\sum_{n=1}^{\beta} \xi_n\pi_{2,n}(x) \\ &\sim \sum_{m=1}^\alpha \mathcal{O}(h^{N-m})\mathcal{O}(h^{m})+\sum_{n=1}^{\beta} \mathcal{O}(h^{N-n})\mathcal{O}(h^{n})\\
& \sim \mathcal{O}(h^N). \quad \square
\end{align}

\end{appendices}

\bibliographystyle{spphys}

\end{document}